\documentclass[11pt]{amsart}
\usepackage[hmargin=3cm,vmargin=3cm]{geometry}

\usepackage{combelow} 

\usepackage[latin1,utf8]{inputenc} 
\usepackage[english]{babel}
\usepackage[backend=bibtex,doi=false,url=false,isbn=false,style=alphabetic,maxnames=6]{biblatex}
\usepackage{etex}
\AtBeginBibliography{\small}

\usepackage{xspace,amssymb,amsfonts,euscript,wasysym}
\usepackage{amsthm,amsmath,amscd,stmaryrd,latexsym,youngtab,verbatim,mathrsfs}
\usepackage{palatino, euscript}

\usepackage{graphicx}
\usepackage[all]{xy}
\SelectTips{cm}{}

\usepackage{tikz, tikz-cd}
\usetikzlibrary{matrix, arrows, calc}
\tikzset{frontline/.style={preaction={draw=white,-,line width=6pt}},}  

\usepackage[dvips]{epsfig}
\usepackage{pinlabel}
\usepackage{psfrag}

\RequirePackage{color}
\definecolor{references}{rgb}{0,0,1}

\RequirePackage[pdftex,
colorlinks = true,
urlcolor = references, 
citecolor = references, 
linkcolor = references, 
]
 {hyperref}
 
\newtheorem{thm}{Theorem}[section]
\newtheorem{theorem}[thm]{Theorem}

\newtheorem{lemma}[thm]{Lemma}
\newtheorem{proposition}[thm]{Proposition}

\newtheorem{corollary}[thm]{Corollary}

\newtheorem*{prop*}{Proposition}
\newtheorem*{lemma*}{Lemma}

\theoremstyle{definition}

\newtheorem{definition}[thm]{Definition}

\newtheorem{example}[thm]{Example}

\theoremstyle{remark}
\newtheorem{remark}[thm]{Remark}

\newtheorem{observation}[thm]{Observation}

\numberwithin{equation}{section}

\def\AB{{\mathbf A}}
\def\EB{{\mathbf E}}
\def\FB{{\mathbf F}}
\def\KC{{\mathcal{K}}}
\def\PB{{\mathbf P}}

\def\AS{{\EuScript A}}
\def\BS{{\EuScript B}}
\def\CS{{\EuScript C}}
\def\IS{{\EuScript I}}
\def\MS{{\EuScript M}}

\def\a{\alpha}
\def\b{\beta}

\def\d{\delta}

\def\e{\varepsilon}

\let\phi=\varphi

\usepackage{bbm}
\def\C{{\mathbbm C}}

\def\R{{\mathbbm R}}
\def\Z{{\mathbbm Z}}

\def\1{\mathbbm{1}}
\newcommand{\one}{\1}

\renewcommand{\sl}{\mathfrak{sl}}
\newcommand{\TL}{\operatorname{TL}}
\renewcommand{\k}{\mathbbm{k}}

\newcommand{\smMatrix}[1]{\left[\begin{smallmatrix}#1\end{smallmatrix}\right]}

\newcommand{\sqmatrix}[1]{\left[\begin{matrix} #1\end{matrix}\right]}
\newcommand{\im}{\operatorname{im}}

\newcommand{\Kar}{\textbf{Kar}}
\newcommand{\Hom}{\operatorname{Hom}}
\newcommand{\End}{\operatorname{End}}
\newcommand{\Id}{\operatorname{id}}
\newcommand{\inv}{^{-1}}

\newcommand{\Ch}{\operatorname{Ch}}
\newcommand{\Cone}{\operatorname{Cone}}
\newcommand{\FT}{\operatorname{FT}}
\newcommand{\TLC}{\mathcal{T}\!\mathcal{L}}
\newcommand{\Bim}{{\rm Bim }}
\newcommand{\SBim}{\mathbb{S}\Bim}

\renewcommand{\Id}{\operatorname{id}}
\renewcommand{\Kar}{\operatorname{Kar}}
\begin{document}

\newcommand{\Sq}{\operatorname{Sq}}
\newcommand{\summand}{\buildrel\oplus\over\subset}
\newcommand{\prefix}{pre}

\newcommand{\tw}{\operatorname{tw}}

\title{Constructing categorical idempotents}

\begin{abstract}We give a general construction of categorical idempotents which recovers the categorified Jones-Wenzl projectors, categorified Young symmetrizers, and other constructions as special cases.  The construction is intimately tied to cell theory in the sense of additive monoidal categories.\end{abstract}

\author{Matthew Hogancamp} \address{Northeastern University, Boston}

\maketitle

\setcounter{tocdepth}{1}
\tableofcontents

\section{Introduction}
\label{sec:intro}
In recent years, categorical idempotents have become important in many aspects of higher representation theory and homology theories in low-dimensional topology. For instance, the categorified Jones-Wenzl projectors and their relatives \cite{CK12a,Roz10a} allow one to categorify Reshetikhin-Turaev invariants of colored links and tangles in $\R^3$ and $\R^2\times[0,1]$.  More recently, the categorified Young idempotents of the author and Ben Elias \cite{ElHogFTdiag-pp} (see also \cite{HogSym-GT,AbHog17}) have become  important components of a conjectural relation between Soergel bimodules in type $A$ and Hilbert schemes of points in $\C^2$ \cite{GNR16}.

Many modern constructions produce categorical idempotents by starting with a monoidal category $\AS$ and a finite complex $\FB\in \KC^b(\AS)$ and showing that tensor powers of $\FB$ stabilize.  In many cases the resulting stable limit $\FB^{\otimes \infty}$ is an idempotent complex.  This idea was first proposed and carried out by Rozansky \cite{Roz10a} in his construction of the categorified Jones-Wenzl idempotent, where $\FB=\FT_n$ is the Khovanov complex associated to the full-twist braid.  

This paper began with the observation that in all examples of idempotents constructed as infinite powers (all those known to the author anyway), the potentially very complicated complex $\FB$ can be replaced by a much simpler complex of the form $(C\rightarrow \one)$, supported in only two homological degrees.  Later we realized that such constructions can be expressed equivalently by formulas which are very reminiscent of the usual bar construction in homological algebra (more specifically \emph{triple cohomology}; see \S \ref{ss:remarks}).


Using the techniques in this paper one can reconstruct with ease a variety of categorical idempotents appearing already in the literature \cite{
CK12a,Roz10a,Roz10b,CH12,Rose12,CauClasp,HogSym-GT,AbHog17,CautisRemarks-pp,LibWilIdemp-pp,ElHogFTdiag-pp}.  
\begin{remark}
One should compare our construction of categorical idempotents to the fact that projective resolutions can be summoned into existence with only a simple incantation (``let $\PB$ be...''), but to say or compute anything useful about them is often quite challenging.
\end{remark}

\subsection{Summary}
\label{ss:summary}

Throughout the paper $\k$ denotes a fixed commutative ring and $\AS$ denotes an additive $\k$-linear monoidal category.  The tensor product in $\AS$ is denoted by $\star$, and the monoidal identity is denoted by $\one$.

Section \S \ref{s:cells and counits} concerns some basic notions in additive monoidal categories. First, in \S \ref{ss:cells and ideals} we recall partially ordered sets of (left, right, or two-sided) cells in $\AS$.  


  In \S \ref{ss:counits} we introduce the notion of a \emph{counital object} in $\AS$ (nonstandard terminology), that is to say an object $C\in \AS$ equipped with a morphism $\e:C\rightarrow \one$ (the \emph{counit}) such that $\e\star \Id_C$ and $\Id_C\star \e$ admit right inverses $\Delta_L,\Delta_R:C\rightarrow C\star C$.   Every coalgebra object is counital, and the direct sum of counital objects has the structure of a counital object.    The set of counital objects admits a transitive reflexive relation by declaring $(C_1,\e_1)\leq (C_2,\e_2)$ if there is a morphism $\nu:C_1\rightarrow C_2$ such that $\e_1 = \e_2\circ \nu$.  This relation is equivalent to the left and right cell orders (Lemma \ref{lemma:counitorder}):
\[
(C_1,\e_1)\leq (C_2,\e_2)  \ \ \ \Leftrightarrow \ \ \ C_1\leq_L C_2 \ \ \ \Leftrightarrow\ \ \ C_1\leq_R C_2.
\]
Section \S \ref{s:construction} contains our main techniques for constructing idempotent complexes over $\AS$.  In \S \ref{ss:cxs} we recall some basic notions concerning categories of complexes, mostly for the purpose of setting up notation.  In \S \ref{ss:idemp alg and coalg} we discuss the theory of idempotent algebras and coalgebras. In \S \ref{ss:the construction} we define a pair of complexes $\PB_C,\AB_C\in \Ch^-(\AS)$ associated to any counital object $(C,\e)$ in $\AS$, by the formulas
\[
\PB_C \ \ =  \ \  \cdots \rightarrow C^{\star 3} \rightarrow C^{\star 2}\rightarrow \underline{C},
\]
\[
\AB_C \ \ =  \ \  \cdots \rightarrow C^{\star 3} \rightarrow C^{\star 2}\rightarrow {C}\rightarrow \underline{\one},
\]
in which the differential $C^{\star n+1}\rightarrow C^{\star n}$ in each is an alternating sum of morphisms $\Id^{\star i}\star \e \star \Id^{\star n-i}$.  Our main theorem (Theorem \ref{thmA}) asserts that $\PB_C$ and $\AB_C$ are idempotent up to homotopy with respect to $\star$, and gives a unique characterization of $\PB_C$, $\AB_C$ in terms of $C$.  The reason for the notation $\PB_C,\AB_C$ is explained in Remark \ref{rmk:P and A}.

\begin{remark}
Complexes of the form $\PB_C$ are idempotent coalgebras in the homotopy category $\KC^-(\AS)$, while complexes of the form $\AB_C$ are idempotent algebras in $\KC^-(\AS)$, and there is an ``idempotent decomposition'' $\one\simeq (\AB_C\rightarrow \PB_C)$ (notation explained in the remarks following Definition \ref{def:twist}).  More generally idempotent complexes of the form $\PB_C$ and $\AB_C$ tend to occur as maximal and minimal terms in idempotent decompositions of $\one$, respectively (see Theorem \ref{thm:TL P and A} for an example).
\end{remark}
\begin{remark}
The two-sided bar complex of an algebra occurs as a special case of this construction; see Example \ref{ex:bimod}.
\end{remark}
\begin{remark}\label{rmk:intro proj}
Categorified Jones-Wenzl idempotents \cite{CK12a,Roz10a} (and their relatives) occur as special cases of $\AB_C$, while projectors such as Rozansky's ``minimal projectors'' \cite{Roz10b} and also the projectors from \cite{AbHog17} occur as special cases of $\PB_C$.
\end{remark}

If $C_1\leq C_2$ then there is a \emph{relative idempotent}  $\PB_{C_2/C_1}$ satisfying
\[
\PB_{C_2}\star \AB_{C_1} \  \ \simeq \ \ \PB_{C_2/C_1} \ \ \simeq \ \ \AB_{C_1}\star \PB_{C_2},
\]
introduced and studied in \S \ref{ss:partial order}.  Theorem \ref{thm:rel idemp thm} establishes the basic properties of these relative idempotents.

\begin{remark}
The categorified Temperley-Lieb idempotents \cite{CH12} occur as special cases of $\PB_{C_2/C_1}$; these include the idempotents from Remark \ref{rmk:intro proj} above as special cases. See \S \ref{ss:cat TL} for more.
\end{remark}

\begin{remark}
The complexes  $\PB_{C_2/C_1}\simeq \PB_{C_2}\star \AB_{C_1}$ are neither idempotent algebras nor coalgebras; such complexes typically arise as non-extremal terms in idempotent decompositions of $\one$ (see Theorem \ref{thm:TL P and A} for an example).
\end{remark}

Section \S \ref{s:normalized} contains a normalized version of the complexes $\PB_C,\AB_C$, in which a large contractible summand has been eliminated.  This construction is akin to the usual normalized two-sided bar complex.  First, in \S \ref{ss:en} we construct a distinguished idempotent endomorphism $e_n$ of $C^{\star n}$, for any counital object $(C,\e)$.  In \S \ref{ss:normalized idempts} we show how $C^{\star n}$ decomposes as a direct sum of images of $e_k$ for various $k\in \{1,\ldots,n\}$.   Nearly all of these summands cancel in $\PB_C$ and $\AB_C$ (Theorem \ref{thm:normalized}), and we find that
\[
\PB_C \ \ \simeq  \ \  \cdots \rightarrow \im e_3  \rightarrow \im e_2 \rightarrow \im e_1,
\]
\[
\AB_C \ \ \simeq  \ \  \cdots \rightarrow \im e_3  \rightarrow \im e_2 \rightarrow \im e_1\rightarrow \underline{\one}.
\]
Strictly speaking the above homotopy equivalences are to be interpreted in the \emph{Karoubi envelope} of $\AS$, in which we have adjoined the images of all idempotent endomorphisms in $\AS$.

In \S \ref{ss:infinite powers} we show (Theorem \ref{thm:powers}) that $\AB_C$ is homotopy equivalent to the ``infinite tensor power'' $\Cone(C\rightarrow \one)^{\star \infty}$.

\begin{remark}\label{rmk:milnor construction}
The fact that $\AB_C$ can be expressed as the ``infinite power'' of $(C\rightarrow \underline{\one})$ was inspired Milnor's ``infinite join'' construction of $EG$ of a topological group $G$ \cite{MilnorUnivBun1}. 
\end{remark}

Finally, \S \ref{s:examples} discusses some examples of using our construction in practice.  In \S \ref{ss:recognition} we discuss how one recognizes complexes of the form $\PB_C$ and $\AB_C$, and \S \ref{ss:cat TL} shows how, using these ideas, one can recover all the categorified Temperley-Lieb idempotents from \cite{CH12}; see Theorem \ref{thm:TL P and A}.

\subsection{Some remarks on the construction}
\label{ss:remarks}

There is a dual version of this story, in which the notion of counital object is dualized, obtaining the notion of unital objects $(A,\eta:\one\rightarrow A)$.  The resulting theory produces a pair of complementary idempotent complexes in $\Ch^+(\AS)$.  We ignore this dual picture entirely, as it can be obtained  from that presented here by passing to the opposite category $\AS^{\text{op}}$.

Any coalgebra object $C\in \AS$ gives rise to a(n augmented) simplicial object in $\AS$.  The complex $\AB_C$ is just the complex corresponding to the augmented simplicial object $C$ under the Dold-Kan correspondence (and $\PB_C$ is the complex corresponding to the associated simplicial object, forgetting the augmentation).  Moreover, our complexes are very closely related to the notion of \emph{triple cohomology}  \cite{Beck-thesis,BarrBeck-HSC}.  Note, however, that our complexes $\PB_C,\AB_C$ are defined even when $C$ is not a coalgebra.  At the same time, triple cohomology is defined for coalgebra objects in categories which are not necessarily linear; so the theory of triple cohomology cannot be expressed in the language of the projectors $\PB_C,\AB_C$ in general.

\begin{remark}
To really make the connection with triple cohomology one should fix a $\k$-linear category $\MS$ on which $\AS$ acts, say, on the right.  If $C\in \AS$ is a coalgebra object then $-\star C$ is a comonad acting on $\MS$, and the cohomology of $C$ with coefficients in $X\in \MS$ can be defined as the homology of the hom complex $\Hom_{\Ch(\MS)}(X\star \PB_C, X)$ with algebra structure inherited from a homotopy equivalence
\[
\Hom_{\Ch(\MS)}(X\star \PB_C, X) \ \simeq \ \End_{\Ch(\MS)}(X\star \PB_C).
\]
We omit the details, as our focus is on the complexes $\PB_C$ themselves, and not on the cohomology theories they represent.  
\end{remark}

\subsection*{Acknowledgements}
The author was supported by NSF grant DMS 1702274.  We thank E.~Gorsky and V.~Mazorchuk for comments on an earlier version.

\section{Cells and counital objects}
\label{s:cells and counits}

\subsubsection{Notation}
Throughout the paper $\k$ denotes a commutative ring and $\AS$ denotes a chosen $\k$-linear, additive, monoidal category.  The tensor product in $\AS$ will be denoted $\star$, and the monoidal identity in $\AS$ will be denoted $\one\in \AS$.  


\subsection{Cells in a monoidal category}
\label{ss:cells and ideals}
Given objects $X,Y\in \AS$ we say that $X$ is a \emph{retract} of $Y$ if $\Id_X$ factors through $\Id_Y$, i.e.~we can write
\[
\Id_X = \pi\circ \Id_Y\circ \sigma
\]
for some morphisms $X\buildrel \sigma\over\rightarrow Y\buildrel\pi\over \rightarrow X$.  In this case $\pi$ is called a \emph{retraction} and $\sigma$ its \emph{section}.  An object $Y$ is \emph{indecomposable} in $\AS$ if every retraction $Y\rightarrow X$ is an isomorphism.
\begin{remark}
This definition of indecomposability is only sensible when $\AS$ is idempotent complete.  A more reasonable definition of indecomposability in general would be, say, $\End_\AS(Y)$ has no nontrivial idempotents.
\end{remark}

Following \cite[\S 4.1]{mazorchuk_miemietz_2011} we write $B_1\leq_L B_2$ (respectively $B_1\leq_R B_2$) if $B_1$ is a retract of $X\star B_2$ (respectively $B_2\star X$) for some $X\in \AS$.  We write $B_1\leq_{LR} B_2$ if $B_1$ is a retract of $X\star B_2\star Y$ for some $X,Y\in \AS$.   The relations $\leq_L$, $\leq_R$, $\leq_{LR}$ are referred to as the left, right, and two-sided cell orders.  Each of these relations is reflexive and transitive, but not anti-symmetric.  We write $B_1\sim_L B_2$ if $B_1\leq_L B_2$ and $B_2\leq_L B_1$ (and similarly for $\sim_R$ and $\sim_{LR}$). Then $\sim_L$, $\sim_R$, and $\sim_{LR}$ are equivalence relations on the set of objects of $\AS$.

A \emph{left cell} (resp.~\emph{right cell}, resp.~\emph{two-sided cell}) in $\AS$ is by definition an equivalence class with respect to $\sim_L$ (resp.~$\sim_R$, resp.~$\sim_{LR}$) generated by an indecomposable object.

A \emph{left tensor ideal} in $\AS$ is a full subcategory $\IS\subset \AS$ with the property that $X\in \AS$ and $B\in \IS$ implies $X\star B\in \IS$.   The notions of \emph{right} and \emph{two-sided} \emph{tensor ideal} are defined similarly.  If $B\in \AS$ then we let $X\star\AS\subset \AS$ denote the right tensor ideal generated by $B$, i.e.~the full subcategory consisting of objects of the form $B\star X$ with $X\in \AS$.  Expressions such as  $\AS \star X$, $\AS\star X\star \AS$, and so on, are defined similarly.

A full subcategory $\BS\subset \AS$ will be called \emph{thick} if it is closed under taking retracts (that is to say, if $Y\in \BS$ and $\Id_X=\pi\circ \Id_Y\circ \sigma$ then $X\in \BS$).   If $\BS\subset \AS$ is a full subcategory, then we let $\overline{\BS}\subset \AS$ denote the smallest thick full subcategory which contains $\BS$.

Note that $B_1\leq_L B_2$ if and only if $\overline{\AS\star B_1}\subset \overline{\AS\star B_2}$, and similarly for $\leq_R$ and $\leq_{LR}$.  In this way, the cell theory of $\AS$ is essentially the study of thick tensor ideals in $\AS$.

\subsection{Counital objects}
\label{ss:counits}
\begin{definition}\label{def:counit}
Let $C\in \AS$ be an object equipped with a morphism $\e:C\rightarrow \one$.  We say that $\e$ is a \emph{counit} if $\e\star \Id_C$ and $\Id_C\star \e$ admit right inverses.   In this case the pair $(C,\e)$ will be called a \emph{counital object}.
\end{definition}


\begin{definition}\label{def:counit order}
Define a transitive reflexive relation $\leq$ on the set of counital objects by declaring $(C_1,\e_1)\leq (C_2, \e_2)$ if $\e_1$ factors through $\e_2$, i.e.~there exists $\nu:C_1\rightarrow C_2$ such that $\e_1 = \e_2\circ \nu$.  We write $C_1\sim C_2$ if $C_1\leq C_2$ and $C_2\leq C_1$.
\end{definition}

\begin{remark}
The monoidal identity $\one\in \AS$ is the unique maximum with respect to $\leq$, and $0\in \AS$ is the unique minimum.
\end{remark}

\begin{definition}\label{def:Cproj}
Let $(C,\e)$ be a counital object in $\AS$.  An object $X\in \AS$ is said to be \emph{right $C$-projective} if there exists a morphism $a:X\rightarrow X\star C$ such that the composition
\[
\begin{tikzpicture}
\tikzstyle{every node}=[font=\small]
\node (a) at (0,0) {$X$};
\node (b) at (2,0) {$X\star C$};
\node (c) at (5,0) {$X\star \one$};
\node (d) at (7,0) {$X$};
\path[->,>=stealth',shorten >=1pt,auto,node distance=1.8cm]
(a) edge node {$a$} (b)
(b) edge node {$\Id_X\star \e$} (c)
(c) edge node {$\cong$} (d);
\end{tikzpicture}
\]
equals $\Id_X$.  \emph{Left $C$-projective objects} are defined similarly.
\end{definition}
Note that right $C$-projective objects form a left tensor ideal in $\AS$.

\begin{lemma}\label{lemma:X leq C}
Let $(C,\e)$ be a counital object.  Then $X\in \AS$ is right (resp.~left) $C$-projective if and ony if $X\leq_L C$ (resp.~$X\leq_R C$).
\end{lemma}
\begin{proof}
Clearly if $X$ is right $C$-projective then $\Id_X$ factors through $\Id_{X\star C}$, so $X\leq_L C$.

Conversely, suppose that $X\leq_L C$, so that there exist an object $Y\in \AS$ and maps $X \buildrel\sigma\over \rightarrow Y\star C \buildrel \pi\over\rightarrow X$ such that $\pi\circ \sigma  = \Id_X$.   Define $a:X\rightarrow X\star C$ to be the composition
\[
\begin{tikzpicture}
\tikzstyle{every node}=[font=\small]
\node (a) at (-.2,0) {$X$};
\node (b) at (2,0) {$Y\star C$};
\node (c) at (5,0) {$Y\star C\star C$};
\node (d) at (8,0) {$X\star C$};
\path[->,>=stealth',shorten >=1pt,auto,node distance=1.8cm]
(a) edge node {$\sigma$} (b)
(b) edge node {$\Id_Y\star \Delta_R$} (c)
(c) edge node {$\pi\star \Id_C$} (d);
\end{tikzpicture}
\]
where $\Delta_R:C\rightarrow C\star C$ is a right inverse to $\Id_C\star \e:C\star C\rightarrow C\star \one\cong C$.  Then $a$ satisfies the condition of Definition \ref{def:Cproj}, so $X$ is right $C$-projective.
\end{proof}

\begin{lemma}\label{lemma:counitorder}
Let $(C_i,\e_i)$, $i=1,2$, be counital objects in $\AS$.  The following are equivalent:
\begin{enumerate}
\item $C_1\leq_L C_2$.
\item $C_1\leq_R C_2$.
\item $C_1\leq C_2$ in the sense of Definition \ref{def:counit order}.
\end{enumerate}
\end{lemma}

Consequently the partial order $C_1\leq C_2$ from Definition \ref{def:counit order} depends only on the thick left (or right) tensor ideals generated by the objects $C_i$.  In particular the counits $\e_i$ are irrelevant. 
\begin{proof}
Throughout the proof we will omit all occurrences of the unitor isomorphisms $\one\star Y\cong Y\cong Y\star \one$.

(1) $\Rightarrow$ (3).  Assume that $C_1\leq_L C_2$.  Then by Lemma \ref{lemma:X leq C} we can find a morphism $f:C_1\rightarrow C_1\star C_2$ such that $(\Id_{C_1}\star \e_2)\circ f = \Id_{C_1}$.  Define
\[
\nu:= (\e_1\star \Id_{C_2})\circ f : C_1  \rightarrow C_2
\]
This $\nu$ satisfies $\e_2\circ \nu = \e_1$, hence (1) implies (3).

(3) $\Rightarrow$ (1).  Suppose that $\nu: C_1\rightarrow C_2$ is such that  $\e_1 = \e_2\circ \nu$.  It is an easy exercise to verify that
\[
\Id_{C_1} = (\Id_{C_1}\star \e_1)\circ f = (\Id_{C_1}\star \e_2)\circ (\Id_{C_1}\star \nu) \circ f,
\]
so $\Id_{C_1}\star \e_2:C_1\star C_2 \rightarrow  C_1$ has a right inverse given by $(\Id_{C_1}\star \nu) \circ f$, which proves that $C_1\leq_L C_2$.  This gives the equivalence (1) $\Leftrightarrow$ (3).  A similar argument proves (2) $\Leftrightarrow$ (3).
\end{proof}

\begin{remark}
Note that $(C,\e_1)\sim (C,\e_2)$ whenever $e_1,\e_2$ are counits with the same underlying object.  Thus the equivalence class of $(C,\e)$ depends only on $C$, and not on $\e$.
\end{remark}

\section{Constructing idempotents}
\label{s:construction}
This section introduces our main techniques for constructing categorical idempotents (that is to say, complexes over $\AS$ which are idempotent with respect to $\star$, up to homotopy).  We begin by setting up some notation.

\subsection{Complexes}
\label{ss:cxs}
If $\AS$ is a $\k$-linear category, then we let $\Ch(\AS)$ denote the dg category of complexes over $\AS$.  Our convention for complexes is such that differentials increase degree by 1, as in
\[
\cdots \buildrel \d\over \rightarrow X^k \buildrel \d\over \rightarrow X^{k+1} \buildrel \d\over \rightarrow \cdots.
\]
Typically the differential of a complex is regarded as implicit, and we often write a complex $(X,\d_X)$ simply as $X$.  Objects of $\AS$ will be regarded as complexes in degree zero.
\begin{remark}
We use the letters $X,Y,Z$ to denote objects of $\Ch(\AS)$.  If we need to examine the chain groups of a complex, we will use the notation $X^k$ ($k\in \Z$).   The notation $\d_X$ will always mean the differential of $X$.

Since objects of $\AS$ are regarded as a special kinds of complexes, the same letters $X,Y,Z$ may also be used to denote objects of $\AS$.
\end{remark}

Morphism spaces in $\Ch(\AS)$ are the complexes
\[
\Hom_{\Ch(\AS)}^l(X,Y) \ := \ \prod_{i\in \Z} \Hom_\AS(X^i, Y^{i+l})
\]
with differential given by the super-commutator
\[
d_{\Ch(\AS)}(f) \ := \ \d_Y\circ f - (-1)^{|f|} f\circ \d_X
\]
where $\d_X$ and $\d_Y$ denote the differentials of $X,Y$, and $|f|\in \Z$ denotes the degree of a homogeneous morphism $f\in \Hom_{\Ch(\AS)}(X,Y)$.

A morphism $f\in \Hom_{\Ch(\AS)}(X,Y)$ is \emph{closed} if $d_{\Ch(\AS)}(f)=0$, and two closed morphisms $f,g \in \Hom_{\Ch(\AS)}^l(X,Y)$ are \emph{homotopic}, written $f\simeq g$, if $f-g=d_{\Ch(\AS)}(h)$ for some $h \in \Hom_{\Ch(\AS)}^{l-1}(X,Y)$.  A complex $X$ \emph{contractible} if $\Id_X\simeq 0$, and more generally complexes $X,Y$ are \emph{homotopy equivalent}, written $X\simeq Y$, if there exist closed morphisms $f:X\leftrightarrow Y: g$ such that $\Id_X-g\circ f\simeq 0$ and $\Id_Y- f\circ g\simeq 0$.

We let $\KC(\AS)=H^0(\Ch(\AS))$ denote the cohomology category of $\Ch(\AS)$, also known as the \emph{homotopy category of complexes}; objects of this category are complexes, and morphisms are degree zero closed morphisms modulo those which are homotopic to zero.  Superscripts $+,-,b$ denote full subcategories of complexes which are bounded below, bounded above, respectively bounded.   For instance $\Ch^-(\AS)\subset \Ch(\AS)$ denotes the full dg subcategory consisting of complexes $X$ with $X^k=0$ for $k\gg 0$.

If $X=(X,\d_X)$ is a complex and $k\in \Z$, then we let $X[k]$ denote the complex with $X[1]^i = X^{i+k}$ and $\d_{X[k]} = (-1)^k\d_X$.

\begin{definition}\label{def:twist}
If $X=(X,\d_X)$ is a complex, then any complex of the form $(X,\d_X+\a)$ we be referred to as a \emph{twist} of $X$, written $\tw_\a(X)$.
\end{definition}

If $f:X\rightarrow Y$ is a degree zero chain map then the \emph{mapping cone} is
\[
\Cone(f)  \ \ := \ \ \tw_{\smMatrix{0&0\\f&0}} (X[1]\oplus Y).
\]

If $\d\in \Hom_{\Ch(\AS)}^1(Z,U)$ is a closed degree 1 morphism then we write
\[
(Z\buildrel\d\over\rightarrow U) \ \ := \ \ \tw_{\smMatrix{0&0\\\d&0}} (Z\oplus U).
\]
So the mapping cone of a chain map $f:X\rightarrow Y$ may be indicated diagrammatically by
\[
\Cone(f) \ \ = \ \ (X[1]\buildrel f \over \rightarrow Y).
\]

Since $\AS$ is monoidal, $\Ch^-(\AS)$ inherits the structure of a monoidal category by
\[
(X\star Y)^k := \bigoplus_{i+j=k} X^i \star Y^j,\qquad\qquad \d_{X\star Y} = \d_X\star \Id_Y+ \Id_X\star \d_Y.
\]
The tensor product of morphisms $f\in \Hom_{\Ch^-(\AS)}(X,X')$ and $g\in \Hom_{\Ch^-(\AS)}(Y,Y')$ is defined using the Koszul sign rule:
\[
(f\star g)|_{X^i\star Y^j} = (-1)^{|g|i} f|_{X^i}\star g|_{Y^j}.
\]
Note that twisting is compatible with the tensor product in the sense that
\[
\tw_\a(X)\star \tw_\b(Y) \cong \tw_{\a\star \Id + \Id\star \b} (X\star Y),
\]
and suspension is compatible with the tensor product in the sense that
\[
X[k]\star Y[l]\cong (X\star Y)[k+l].
\]
Combining these, we see that mapping cones interact with tensor product according to
\[
Z\star \Cone(f) \ \cong \ \Cone(\Id_Z\star f),\qquad\qquad \Cone(f)\star Z\ \cong \ \Cone(f\star \Id_Z).
\]

\subsection{Idempotent algebras and coalgebras}
\label{ss:idemp alg and coalg}
The notion of idempotent (co)algebra makes sense in any monoidal category.  Below we are only interested in idempotent coalgebras in $\KC^-(\AS)$, and we restrict to this setting for concreteness.  The unpublished note \cite{BoyDrin-idemp} is an excellent read, and discusses idempotent coalgebras outside of the triangulated or dg setting.
\begin{definition}\label{def:idemp alg}
An \emph{idempotent coalgebra} in $\KC^-(\AS)$ is a complex $\PB\in \KC^-(\AS)$ equipped with a degree zero chain map $\boldsymbol{\e}:\PB\rightarrow \one$ such that 
\[
\PB\star \Cone(\boldsymbol{\e})\ \simeq \ 0 \ \simeq \ \Cone(\boldsymbol{\e})\star \PB
\]
(equivalently, $\Id_{\PB}\star \boldsymbol{\e}$ and $\boldsymbol{\e}\star \Id_{\PB}$ are homotopy equivalences). 

An \emph{idempotent algebra} in $\KC^-(\AS)$ is a complex $\AB\in \KC^-(\AS)$ equipped with a degree zero chain map $\boldsymbol{\eta}:\one\rightarrow \AB$ such that 
\[
\AB\star \Cone(\boldsymbol{\eta})\ \simeq \ 0 \ \simeq \ \Cone(\boldsymbol{\eta})\star \AB
\]
(equivalently, $\Id_{\AB}\star \boldsymbol{\eta}$ and $\boldsymbol{\eta}\star \Id_{\AB}$ are homotopy equivalences). 

A \emph{2-step idempotent decomposition} of $\one\in \AS$ is  homotopy equivalence
\[
\one \ \simeq \ \left(\AB\buildrel \d\over \rightarrow \PB\right)
\]
in which $\PB\star \AB\simeq 0 \simeq \AB\star \PB$.  Here the notation is as in the discussion following definition \ref{def:twist}; in particular $\d$ is a degree one closed morphism $\d\in \Hom^1_{\Ch(\AS)}(\AB,\PB)$.  We also refer to $\AB$ and $\PB$ as \emph{complementary idempotents}, and write $\AB \simeq \PB^c$, and $\PB\simeq \AB^c$.
\end{definition}

\begin{remark}
It is a nontrivial consequence of the definitions that if $(\PB,\boldsymbol{\e})$ is an idempotent coalgebra then $\boldsymbol{\e}\star \Id_\PB\simeq \Id_{\PB}\star \boldsymbol{\e}$, and their common homotopy inverse $\boldsymbol{\Delta}:\PB\rightarrow \PB\star \PB$ is coassociative up to homotopy.  This justifies our referring to $(\PB,\boldsymbol{\e})$ as an idempotent coalgebra.  Similar remarks apply to idempotent algebras. 
\end{remark}

Next we state without proof some basic facts concerning idempotent (co)algebras.  Proofs can be found in \cite{BoyDrin-idemp,Hog17a}.  Below, let $\one\simeq (\AB\rightarrow \PB)$ be an idempotent decomposition of identity.

\begin{observation}
Any idempotent coalgebra $(\PB,\e)$ determines a 2-step idempotent decomposition of identity by taking $\AB:=\Cone(\boldsymbol{\e})$.  Any idempotent algebra $(\AB,\boldsymbol{\eta})$ determines a 2-step idempotent decomposition of identity by taking $\PB:=\Cone(\boldsymbol{\eta})[-1]$.

Conversely, if $\one\simeq (\AB\rightarrow \PB)$ is an idempotent decomposition of identity then $\PB$ and $\AB$ are naturally equipped with the structures of an idempotent coalgebra and algebra, respectively.
\end{observation}

\begin{observation}
For any $X,Y\in \KC^-(\AS)$ we have $\PB\star X\simeq X$ iff $\AB\star X\simeq 0$, and $\PB\star Y\simeq 0$ iff $\AB\star Y\simeq Y$.  Similar statements hold with $X,Y$ on the left.
\end{observation}

\begin{observation}
The following full subcategories of $\KC^-(\AS)$ are closed under mapping cones, direct sum, and suspension:
\begin{enumerate}
\item $\{X\in \KC^-(\AS)\ |\ \PB\star X\simeq X\}$ (and similarly with $X$ on the left).
\item $\{X\in \KC^-(\AS)\ |\ \PB\star X\simeq 0\}$ (and similarly with $X$ on the left).
\item $\{X\in \KC^-(\AS)\ |\ \PB\star X\simeq X\star \PB\}$.
\item $\{X\in \KC^-(\AS)\ |\ \AB\star X\simeq X\star \AB\}$.
\item $\{X\in \KC^-(\AS)\ |\ \AB\star X\star \PB\simeq 0 \simeq \PB\star X\star \AB\}$.
\end{enumerate}
In fact, the categories (3), (4), and (5) coincide.
\end{observation}

\begin{observation}
The following are equivalent:
\begin{enumerate}
\item  $\PB\star X\simeq X\star \PB$ for all $X\in \KC^-(\AS)$.
\item $\AB\star X\simeq X\star \AB$  for all $X\in \KC^-(\AS)$.
\item $\{X\in \KC^-(\AS)\ |\ \PB\star X\simeq X\} \ \ = \ \ \{X\in \KC^-(\AS)\ |\ X\star \PB\simeq X\}$.
\item $\{X\in \KC^-(\AS)\ |\ \PB\star X\simeq 0\} \ \ = \ \ \{X\in \KC^-(\AS)\ |\ X\star \PB\simeq 0\}$
\end{enumerate}
In fact, the homotopy equivalences from (1) and (2) can be chosen to be natural in $X$, up to homotopy.
\end{observation}

Note that $(\PB,\boldsymbol{\e})$ may be regarded as a counital object in the homotopy category $\KC^-(\AS)$, so the transitive reflexive relation $\leq$ from Definition \ref{def:counit order} applies.  The following is a strengthening of Lemma \ref{lemma:counitorder} in the context of idempotent coalgebras (compare with Theorem 4.24 in \cite{Hog17a}).  It is well-known to experts.

\begin{proposition}\label{prop:idemp coalg order}
Let $(\PB_i,\boldsymbol{\e}_i)$ be idempotent coalgebras ($i=1,2$). The following are equivalent.
\begin{enumerate}
\item $\PB_1\leq \PB_2$.
\item $\PB_1\star \PB_2\simeq \PB_1$.
\item $\PB_2\star \PB_1\simeq \PB_1$.
\end{enumerate}
\end{proposition}
\begin{proof}
Throughout the proof we regard $\PB_i$ as counital objects in the additive monoidal category $\KC^-(\AS)$,  so that the results of \S \ref{ss:counits} are available. Let $\AB_i:=\PB_i^c$ be the complementary idempotents.

Observe that $\PB_1\leq \PB_2$ implies $\PB_1\leq_L \PB_2$ (Lemma \ref{lemma:counitorder}) which implies $\PB_1$ is a retract of $\PB_1\star \PB_2$ (Lemma \ref{lemma:X leq C}).  This, in turn implies that $\PB_1\star \AB_2$ is a retract of $\PB_1\star \PB_2\star \AB_2$, which is contractible.  Thus, (1) implies $\PB_1\star \AB_2\simeq 0$, which is equivalent to (2).  Conversely, (2) clearly implies $\PB_1\leq_L \PB_2$, which implies (1) by Lemma \ref{lemma:counitorder}.

A similar argument establishes the equivalence (1) $\Leftrightarrow$ (3).
\end{proof}

\begin{corollary}\label{cor:uniqueness}
An idempotent coalgebra $\PB\in \KC^-(\AS)$ is uniquely determined up to homotopy equivalence by either of the following full subcategories of $\KC^-(\AS)$:
\begin{enumerate}
\item $\{X\in \KC^-(\AS)\ |\ \PB\star X\simeq X\}$.
\item $\{X\in \KC^-(\AS)\ |\ \PB\star X\simeq 0\}$.
\item $\{X\in \KC^-(\AS)\ |\ X\star \PB \simeq X\}$.
\item $\{X\in \KC^-(\AS)\ |\  X\star \PB\simeq 0\}$.
\end{enumerate}
\end{corollary}
\begin{proof}
Suppose $\PB$ and $\PB'$ are idempotent coalgebras in $\KC^-(\AS)$ satisfying $\PB\star X\simeq X$ if and only if $\PB'\star X\simeq X$, for all $X\in \KC^-(\AS)$.  Then $\PB'\star \PB'\simeq \PB'$ implies $\PB\star \PB'\simeq \PB'$, which is equivalent to $\PB'\star \PB\simeq \PB'$ by  Proposition \ref{prop:idemp coalg order}.  By symmetry we also have $\PB'\star \PB\simeq \PB$, hence $\PB\simeq \PB'$.    This proves that $\PB$ is uniquely determined by the full subcategory (1).  A similar argument takes care of the remaining cases (though for (2) and (4) it is necessary to use the version of Proposition \ref{prop:idemp coalg order} which applies to idempotent algebras; details are left to the reader).
\end{proof}

\subsection{The categorical idempotents associated to $C$}
\label{ss:the construction}

Let $(C,\e)$ be a counital object in $\AS$, and define the following complex in $\Ch^-(\AS)$:
\begin{equation}\label{eq:introbarP}
\PB_{C} \ := \ \ \ \ 
\begin{tikzpicture}[baseline=-.2em]
\tikzstyle{every node}=[font=\small]
\node (a) at (0,0) {$\underline{C}$};
\node (b) at (-2,0) {$C^{\star 2}$};
\node (c) at (-4,0) {$C^{\star 3}$};
\node (d) at (-6,0) {$\cdots$};
\path[->,>=stealth',shorten >=1pt,auto,node distance=1.8cm]
(b) edge node {$\e \Id -\Id\e$} (a)
(c) edge node {} (b) 
(d) edge node {} (c);
\end{tikzpicture},
\end{equation}
where we have underlined the term in degree zero, and the differential $C^{\star r+1}\rightarrow C^{\star r}$ is the alternating sum $\sum_{i=0}^r (-1)^i \Id^{\star i}\star \e\star \Id^{\star r-i}$.

Note that $\PB_C$ has a chain map $\PB_C\rightarrow \one$ (given by $\e$ in degree zero).  We let $\AB_C:=\Cone(\PB_C\rightarrow \one)$.  That is to say,
\begin{equation}\label{eq:introbarA}
\AB_C \ := \ \ \ \ 
\begin{tikzpicture}[baseline=-.2em]
\tikzstyle{every node}=[font=\small]
\node (a) at (0,0) {$\underline{\one}$};
\node (b) at (-2,0) {$C$};
\node (c) at (-4.5,0) {$C^{\star 2}$};
\node (d) at (-10,0) {$\cdots $};

\path[->,>=stealth',shorten >=1pt,auto,node distance=1.8cm]
(b) edge node {$\e$} (a)
(c) edge node {$-(\e \Id -\Id\e)$} (b)
(d) edge node {$-(\e \Id\Id - \Id \e \Id + \Id\Id\e)$} (c);
\end{tikzpicture}.
\end{equation}

\begin{lemma}\label{lemma:A kills C}
We have $\AB_C\star C\simeq 0\simeq C\star \AB_C$.
\end{lemma}
\begin{proof}
Below we define an explicit homotopy which realizes $C\star \AB_{C}\simeq 0$.  First, let $\Delta_R:C\rightarrow C\star C$ be a right inverse to $\Id_C\star \e: C\star C\rightarrow C\star \one\cong C$,  which exists since $\e:C\rightarrow \one$ is a counit.

For each $k\geq 0$ let $h^k: C^{\star k+1}\rightarrow C^{\star k+2}$ denote $\Delta_R\star \Id^{\star k}$, and let $d^{k}:C^{\star k+1}\rightarrow C^{\star k}$ denote $\sum_{i=1}^k (-1)^{i-1} \Id^{\star i} \star \e\star \Id^{\star k-i}$.  Then $d^k$ are the components of the differential of $C\star \AB_{C}$.  It is straightforward to verify that 
\[
d^{k+1}\circ h^k + h^{k-1}\circ d^k = \Id_{C^{\star k+1}}
\]
for all $k\geq 0$, where $h^{-1}=0$ by convention.  This shows that $C\star \AB_C\simeq 0$.  The proof  that $\AB_C\star C\simeq 0$ is similar.
\end{proof}

Our first main result states that $(\PB_C,\AB_C)$ are complementary idempotents in $\KC^{-}(\AS)$.

Before stating, recall that an object $X\in \AS$ is left (rep.~right) $C$-projective iff $X$ is in the thick right (resp.~left) tensor ideal generated by $C$, by Lemma \ref{lemma:X leq C}.
\begin{thm}
\label{thmA}
Let $(C,\e)$ be a counital object in $\AS$. Then: 
\begin{enumerate}
\item  $\PB_C$ is an idempotent coalgebra in $\KC^-(\AS)$ and $\AB_C$ is the complementary idempotent algebra.
\item $X\in \KC^-(\AS)$ satisfies $\PB_C\star X\simeq X$, equivalently $\AB_C\star X\simeq 0$, if and only if $X$ is homotopy equivalent to a complex of left $C$-projectives.
\item $X\in \KC^-(\AS)$ satisfies $X\star \PB_C \simeq X$, equivalently $X\star \AB_C\simeq 0$, if and only if $X$ is homotopy equivalent to a complex of right $C$-projectives.
\item $\PB_C$ and $\AB_C$ are uniquely characterized up to homotopy equivalence by:
\begin{enumerate}
\item[(U1)] there is a homotopy equivalence $\one\simeq (\AB_C\rightarrow \PB_C)$.
\item[(U2)] $\AB_C\star C\simeq 0$.
\item[(U3)] $\PB_C$ is a complex of left $C$-projectives.
\end{enumerate}
\end{enumerate}
\end{thm}

Before proving we need a technical lemma.

\begin{lemma}\label{lemma:Cannihilation}
Suppose $\AB',\PB'\in \KC^-(\AS)$ satisfy (U1), (U2), (U3) from Theorem \ref{thmA}.  Then $X\in \KC^-(\AS)$ satisfies $\AB'\star X\simeq 0$ if and only if $X$ is homotopy equivalent to a complex of left $C$-projectives.  There is a similar statement for right $C$-projectives.
\end{lemma}
\begin{proof}
First, suppose $X\in \AS$ is left $C$-projective.  Then $X$ is a retract of $C\star X$ by definition, hence $\AB'\star X$ is a retract of $ \AB'\star C\star X$, which is contractible by hypothesis.  This shows that $\AB'\star X\simeq 0$ when $X$ is left $C$-projective.

Now, suppose $X$ is a complex of left $C$-projectives.  Let $X':=\bigoplus_k X^k[-k]$ be the direct sum of the chain objects of $X$.  Then $X'$ is $X$ with zero differential, and conversely $X$ can be regarded as a twist of $X'$:
\[
X \ = \ \tw_\d(X').
\]
Each chain object satisfies $\AB'\star X^k\simeq 0$ by the above paragraph, so it follows that $\AB'\star X'\simeq 0$.  Then $\AB'\star X\simeq 0$ by standard homological perturbation techniques (see Corollary 4.7 with curvature $z=0$ in \cite{Hog-CHPT-pp}).

Now, conversely, suppose $\AB'\star X$.  Then
\[
0\simeq \Cone(\PB'\rightarrow \one)\star X \cong \Cone(\PB'\star X\rightarrow \one\star X),
\]
which implies that $X\simeq \PB'\star X$, which is a complex of left $C$-projectives.  This proves the lemma (the statement about right $C$-projectives follows by symmetry).
\end{proof}

\begin{proof}[Proof of Theorem \ref{thmA}]
Given that $\AB_C\star C\simeq 0\simeq C\star \AB_C$ (Lemma \ref{lemma:A kills C}) and $\PB_C$ is a complex of left and right $C$-projectives, Lemma \ref{lemma:Cannihilation} tells us that $\AB_C\star \PB_C\simeq 0\simeq \PB_C\star \AB_C$.  Statement (1) follows, since $\one\simeq (\AB_C\rightarrow \PB_C)$ (recall the discussion in \S \ref{ss:idemp alg and coalg}).

%

Statements (2) and (3) are immediate consequences of Lemma \ref{lemma:Cannihilation}.

For the uniqueness statement (4), suppose $\PB'$ and $\AB'$ satisfy (U1), (U2), (U3).  Then $\AB'\star \PB_C\simeq 0$ by Lemma \ref{lemma:Cannihilation} applied to $\AB'$.  This implies $\AB'\star \AB_C\simeq \AB'$.  Applying the right $C$-projective version of Lemma \ref{lemma:Cannihilation} to $\AB_C$ yields $\PB'\star \AB_C\simeq 0$, hence $\AB'\star \AB_C\simeq \AB_C$.  This shows $\AB_C\simeq \AB'$; a similar argument shows $\PB_C\simeq \PB'$.
\end{proof}

\begin{remark}
The complexes $\PB_C$ and $\AB_C$ depend only on the thick left (or right) tensor ideal generated by $C$, up to homotopy equivalence by part (2) of the above theorem and Corollary \ref{cor:uniqueness}.  In particular the particular choice of counit $\e:C\rightarrow \one$ is irrelevant.  
\end{remark}

\begin{remark}\label{rmk:P and A}
The notation for $\PB_C$ and $\AB_C$ is meant to remind the reader that certain kinds of projective resolutions occur as special cases of $\PB_C\rightarrow \one$, and in such cases $\AB_C=\Cone(\PB_C\rightarrow \one)$ is the associated acyclic complex (the cone of a quasi-isomorphism).\footnote{Note that $\AS$ is usually not assumed to be an abelian category, only an additive category, so the notions of homology, quasi-isomorphism, and acyclic complexes are not typically defined.  Even when $\AS$ is abelian, we work in a setting where acyclic complexes (i.e.~exact sequences) are not generally isomorphic to zero, but contractible complexes (i.e.~split exact sequences) are.}

An alternate mnemonic: $\PB_C$ \underline{p}reserves $C$ in the sense that $\PB_C\star C\simeq C\simeq C\star \PB_C$ while $\AB_C$ \underline{a}nnihilates $C$.
\end{remark}


\begin{example}\label{ex:bimod}
Let $A$ is a $\k$-algebra, and let $\AS$ denote the category of $A,A$-bimodules with monoidal product $\star = \otimes_A$ and monoidal identity $\one = A$.  Then $C:=A\otimes_\k A$ is a coalgebra object with counit given by the multiplication map $C=A\otimes_\k A \rightarrow A  = \one$ and comultiplication given by the insertion of 1:
\[
C = A\otimes_\k A \rightarrow A\otimes_\k A\otimes_\k A = C^{\star 2},\qquad\qquad a\otimes b\mapsto a\otimes 1\otimes b.
\]
Then $C^{\otimes i} = A^{\otimes_\k i+1}$ and  $\PB_C$ is the usual two-sided bar complex associated to $A$:
\[
\PB_{A\otimes_\k A} \ \ = \ \ 
\begin{tikzpicture}[baseline=-.3em]
\node (a) at (0,0) {$\underline{A\otimes_\k A}$};
\node (b) at (-3,0) {$A\otimes_\k A\otimes A$};
\node (c) at (-6,0) {$A^{\otimes_\k 4}$};
\node (d) at (-9,0) {$\cdots$};
\path[->,>=stealth',shorten >=1pt,auto,node distance=1.8cm]
(b) edge node {} (a)
(c) edge node {} (b)
(d) edge node {} (c);
\end{tikzpicture}
\] and $\AB_{A\otimes_\k A}=\Cone(\PB_{A\otimes_\k A}\rightarrow A)$ is the cone of a quasi-isomorphism.  When $A$ is projective over $\k$, the natural map $\PB_{A\otimes_\k A}\rightarrow A$ is a resolution of $A$ by projective bimodules.
\end{example}

\begin{example}\label{ex:Hopf}
Let $H$ be a Hopf algebra (or bialgebra) over $\k$, and let $\AS$ be the category of $H$-modules, with the monoidal product $\star = \otimes_\k$ and monoidal identity $\one = \k$ (and $H$-actions defined using the comultiplication and counit of $H$, respectively).  Then $C:=H$ is a coalgebra object in $\AS$.  If $H$ is projective as a $\k$-module then $\PB_H\rightarrow \k$ is a resolution of $\k$ by projective $H$-modules (exercise), and $\AB_H = \Cone(\PB_H\rightarrow \k)$ is the cone of this quasi-isomorphism.
\end{example}

The following examples formed the main motivation for this work.

\begin{example}\label{ex:JW}
Let $\TLC_n$ denote the categorification of the Temperley-Lieb algebra $\TL_n$ given by projective modules over Khovanov's ring $H^n$ \cite{Kh02} or, equivalently, Bar-Natan's category of $n,n$-tangles in a rectangle \cite{B-N05}. Let $U_i\in \TLC_n$ denote the object corresponding to the cup / cap tangle\vskip7pt
\[
U_i \ \ = \ \ \begin{minipage}{1.5in}
\labellist
\small
\pinlabel $1$ at 0 -6
\pinlabel $\cdots$ at 15 25
\pinlabel $i$ at 50 -6
\pinlabel $i+1$ at 70 -6
\pinlabel $\cdots$ at 100 25
\pinlabel $n$ at 115 -6
\endlabellist
\includegraphics[scale=.9]{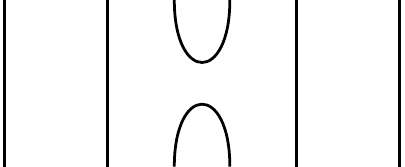}
\end{minipage}\ , \qquad\quad 1\leq i\leq n-1.
\]\vskip10pt
\noindent Then each $U_i$ is a coalgebra object in $\TLC_n$ (after a grading shift).  The direct sum of coalgebras is a coalgebra, hence the direct sum $C:=U_1\oplus \cdots \oplus U_{n-1}$ is a coalgebra.  The corresponding unital idempotent $P_n:=\AB_C\in \KC^-(\TLC_n)$ is the categorified Jones-Wenzl projector of Cooper-Krushkal and Rozansky \cite{CK12a,Roz10a}, up to homotopy equivalence.
\end{example}

The ease with which $\AB_C$ can be constructed should be contrasted with the intricate arguments for constructing categorified Jones-Wenzl projectors in \cite{CK12a}.  We should remark however, that the construction in \cite{CK12a} accomplishes more than can be seen from the general construction, in the form of a beautiful recursion which expresses $P_n\in \KC^-(\TLC_n)$ as an infinite twisted complex involving $P_{n-1}$.  This recursion was exploited in \cite{quasiPaper} to investigate the dg algebra of endomorphisms of $P_n$.

\begin{example}\label{ex:generalized JW}
One can recover \emph{all} the categorified symmetrizers (which might also be referred to as antisymmetrizers, depending on one's mood) that have appeared in literature over the years as special cases of $\AB_C$.  Examples include Rose's $\sl_3$ projectors \cite{Rose12}, Cautis' $\sl_n$ clasps \cite{CauClasp}, and the categorified Young symmetrizers \cite{HogSym-GT}.  These latter complexes are special cases of the so-called antispherical projectors; see below.
\end{example}

\begin{example}\label{ex:antispherical}
Let $W$ be a Coxeter group with finite set of simple reflections $S$, and let $\SBim(W)$ be the category of Soergel bimodules for $W$ (or its diagrammatic version, due to Elias-Williamson \cite{EWsoergelCalc}).  Let $B_s$ be the indecomposable bimodule assocated to a simple reflection $s\in S$.  Then $B_s(-1)$ is a coalgebra in $\SBim(W)$ (while $B_s(1)$ is an algebra).  If we let $C= \bigoplus_{s\in S} B_s(-1)$ then $\AB_C$ is the \emph{categorified antispherical projector}, constructed for finite $W$ in Libedinsky-Williamson \cite{LibWilIdemp-pp} using infinite powers of the ``full twist'' Rouquier complex.  Note that our construction of anti-spherical projectors makes sense for infinite Coxeter groups.
\end{example}

\subsection{Relative idempotents}
\label{ss:partial order}
Let $(C_1,\e_1)\leq (C_2,\e_2)$ be counital objects in $\AS$ (where $\leq$ is as in Definition \ref{def:counit order}).  Let $\nu:C_1\rightarrow C_2$ be a map such that $\e_2\circ \nu = \e_1$.  Then we have a chain map $\phi:\PB_{C_1}\rightarrow \PB_{C_2}$ pictured as follows:
\begin{equation*}\label{eq:connectingmap}
\begin{tikzpicture}[baseline=-2.8em]
\tikzstyle{every node}=[font=\small]
\node (a) at (0,0) {$C_2$};
\node (b) at (-2.5,0) {$C_2^{\star 2}$};
\node (c) at (-5,0) {$C_2^{\star 3}$};
\node (d) at (-7.5,0) {$\cdots$};
\node (e) at (0,-2) {$C_1$};
\node (f) at (-2.5,-2) {$C_1^{\star 2}$};
\node (g) at (-5,-2) {$C_1^{\star 3}$};
\node (h) at (-7.5,-2) {$\cdots$};
\node (i) at (-10,0) {$\PB_{C_2}$};
\node (j) at (-10,-2) {$\PB_{C_1}$};
\node at (-8.8,0) {$=$};
\node at (-8.8,-2) {$=$};
\path[->,>=stealth',shorten >=1pt,auto,node distance=1.8cm]
(j) edge node[left] {$\phi$} (i)
(b) edge node[right] {} (a)
(c) edge node[right] {} (b)
(d) edge node[right] {} (c)
(f) edge node[right] {} (e)
(g) edge node[right] {} (f)
(h) edge node[right] {} (g)
(e)  edge node[right] {$\nu$} (a)
(f)  edge node[right] {$\nu^{\star 2}$} (b)
(g)  edge node[right] {$\nu^{\star 3}$} (c);
\end{tikzpicture}
\end{equation*}
This exends to a map of cones $\psi:\AB_{C_1}\rightarrow \AB_{C_2}$:
\begin{equation*}\label{eq:connectingmap2}
\begin{tikzpicture}[baseline=-2.8em]
\tikzstyle{every node}=[font=\small]
\node (z) at (2.5,0) {$\one$};
\node (a) at (0,0) {$C_2$};
\node (b) at (-2.5,0) {$C_2^{\star 2}$};
\node (c) at (-5,0) {$C_2^{\star 3}$};
\node (d) at (-7.5,0) {$\cdots$};
\node (zz) at (2.5,-2) {$\one$};
\node (e) at (0,-2) {$C_1$};
\node (f) at (-2.5,-2) {$C_1^{\star 2}$};
\node (g) at (-5,-2) {$C_1^{\star 3}$};
\node (h) at (-7.5,-2) {$\cdots$};
\node (i) at (-10,0) {$\AB_{C_2}$};
\node (j) at (-10,-2) {$\AB_{C_1}$};
\node at (-8.8,0) {$=$};
\node at (-8.8,-2) {$=$};
\path[->,>=stealth',shorten >=1pt,auto,node distance=1.8cm]
(j) edge node[left] {$\psi$} (i)
(b) edge node[right] {} (a)
(e) edge node[right] {} (zz)
(a) edge node[right] {} (z)
(c) edge node[right] {} (b)
(d) edge node[right] {} (c)
(f) edge node[right] {} (e)
(g) edge node[right] {} (f)
(h) edge node[right] {} (g)
(zz) edge node[right] {$\Id$} (z)
(e)  edge node[right] {$\nu$} (a)
(f)  edge node[right] {$\nu^{\star 2}$} (b)
(g)  edge node[right] {$\nu^{\star 3}$} (c);
\end{tikzpicture}
\end{equation*}

\begin{lemma}\label{lemma:relprojector}
We have
\[
\Cone(\AB_{C_1}\rightarrow \AB_{C_2})[-1] \ \simeq \ \Cone(\PB_{C_1}\rightarrow \PB_{C_2}) \ \simeq \  \PB_{C_2}\star \AB_{C_1} \ \simeq \  \AB_{C_1}\star \PB_{C_2}.
\]
If $\EB$ denotes any of the equivalent complexes above, then $\EB$ satisfies the additional properties
\begin{itemize}
\item[(a)] $\EB\star \PB_{C_1}\simeq 0 \simeq \PB_{C_1}\star \EB$, 
\item[(b)] $\EB\star \PB_{C_2} \simeq \EB\simeq \PB_{C_2}\star \EB$,
\item[(c)] $\EB\star \EB\simeq \EB$.
\end{itemize}
\end{lemma}
\begin{proof}
The equivalence $\Cone(\AB_{C_1}\rightarrow \AB_{C_2})[-1] \ \simeq \ \Cone(\PB_{C_1}\rightarrow \PB_{C_2})$ is the cancelation of the contractible summand $\Cone(\one\rightarrow \one)$.  Let $\EB:=\Cone(\PB_{C_1}\rightarrow \PB_{C_2})$.

Observe that, since $C_1 \leq C_2$ we have $C_1\leq_L C_2$ and $C_1\leq_R C_2$ by Lemma \ref{lemma:counitorder}, hence
\[
\PB_{C_1}\star \AB_{C_2}\ \simeq \ 0 \ \simeq \ \AB_{C_2}\star \PB_{C_1}
\]
by parts (2), (3) of Theorem \ref{thmA}. Since $\EB\simeq \Cone(\AB_{C_1}\rightarrow \AB_{C_2})[-1]$) and both $\AB_{C_1},\AB_{C_2}$ are annihilated by $\PB_{C_1}$, it follows that $\EB\star \PB_{C_1}\simeq 0\simeq \PB_{C_1}\star \EB$.

Since $\EB\star \PB_{C_1}\simeq 0$, it follows that $\EB\star \AB_{C_1} \simeq \EB$.  Using the expression $\EB\simeq \Cone(\PB_{C_1}\rightarrow \PB_{C_2})$ it follows that
\[
\EB\simeq \EB\star \AB_{C_1} \simeq \Cone(\PB_{C_1}\star \AB_{C_1}\rightarrow \PB_{C_2}\star \AB_{C_1}),
\]
which is homotopy equivalent to $\PB_{C_2}\star \AB_{C_1}$, after canceling the contractible summand $\PB_{C_1}\star \AB_{C_1}$.  A similar argument shows that $\EB\simeq \AB_{C_1}\star \PB_{C_2}$.

Properties (a), (b), (c) are now clear from $\EB\simeq \AB_{C_1}\star \PB_{C_2}\simeq \PB_{C_2}\star \AB_{C_1}$.
\end{proof}

In the statement below, $[C]$ denotes the equivalence class of a $C$ with respect to the the equivalence relation $C\sim C'$ if $C\leq C'$ and $C'\leq C$.
\begin{theorem}\label{thm:rel idemp thm}
Let $(C_i, \e_i)$ be counital objects in $\AS$ ($i=1,2$).  The following are equivalent:
\begin{enumerate}
\item $C_1\leq C_2$ (as counital objects in $\AS$)
\item $\PB_{C_1}\leq \PB_{C_2}$ (as counital objects in $\KC^-(\AS)$).
\item there exists a decomposition
\begin{equation*}\label{eq:rel decomp}
\PB_{C_2} \ \ \simeq \ \ \left(\EB\rightarrow \PB_{C_1}\right).
\end{equation*}
(notation as in \S \ref{ss:cxs}) where  $\EB\in \KC^-(\AS)$ satisfies $\EB\star C_1\simeq 0 \simeq C_1\star \EB$.
\end{enumerate}
Furthermore, if either of these conditions are satisfied, the complex $\EB$ from (3) is uniquely determined by  $[C_1]$ and $[C_2]$ up to homotopy equivalence, and satisfies the additional properties:
\begin{itemize}
\item[(a)] $\EB\star \PB_{C_1}\simeq 0 \simeq \PB_{C_1}\star \EB$, 
\item[(b)] $\EB\star \PB_{C_2} \simeq \EB\simeq \PB_{C_2}\star \EB$,
\item[(c)] $\EB\star \EB\simeq \EB$.
\end{itemize}
\end{theorem}
We think of $\EB$ as the complement of $\PB_{C_1}$ \emph{relative to $\PB_{C_2}$}, and \eqref{eq:rel decomp} describes a decomposition of $\PB_{C_2}$ as a categorical ``sum'' of the two smaller idempotents $\PB_{C_1}$ and $\EB$.


\begin{proof}
The equivalence (1) $\Leftrightarrow$ (2) holds since $(C_1,\e_1)\leq (C_2,\e_2)$ if and only if $C_1\leq_L C_2$ (Lemma \ref{lemma:counitorder}), which holds if and only if $\PB_{C_1}\star \PB_{C_2}\simeq \PB_{C_1}$ (part (2) of Theorem \ref{thmA}), which is equivalent to $\PB_{C_1}\leq \PB_{C_2}$ (Proposition \ref{prop:idemp coalg order}).

The implication (1) $\Rightarrow$ (3) follows by taking $\EB$ to be the cone of the natural map $\PB_{C_1}\rightarrow \PB_{C_2}$ as in Lemma \ref{lemma:relprojector}.

We will show that (3) $\Rightarrow$ (2).  Suppose we have a decomposition $\PB_{C_2}\simeq (\EB\rightarrow \PB_{C_1})$ in which $\EB$ is annihilated by $C_1$ up to homotopy on the left and right.  Then $\EB$ is annihilated by $\PB_{C_1}$ on the left and right, from which it follows that
\[
\PB_{C_1}\star \PB_{C_2} \ \ \simeq \ \ \left(\PB_{C_1}\star \EB\rightarrow \PB_{C_1}\star \PB_{C_1}\right) \ \ \simeq \ \ \PB_{C_1}
\] 
after contracting a contractible summand and using $\PB_{C_1}\star \PB_{C_1}\simeq \PB_{C_1}$, which is equivalent to (2).  This completes the proof that (1)-(3) are equivalent.

We now prove uniqueness of $\EB$.  Suppose we have a decomposition $\PB_{C_2}\simeq (\EB\rightarrow \PB_{C_1})$ in which $\EB$ is annihilated by $C_1$ up to homotopy on the left and right.  Then $\EB\star \PB_{C_1}\simeq 0$, which implies that $\EB\star \AB_{C_1}\simeq \EB$.  Then tensoring the decomposition $\PB_{C_2}\simeq (\EB\rightarrow \PB_{C_1})$ on the right with $\AB_{C_1}$ yields
\[
\PB_{C_2}\star \AB_{C_1} \ \ \simeq \ \ (\EB \star \AB_{C_1} \rightarrow \PB_{C_1}\star \AB_{C_1}) \ \ \simeq  \ \ \EB \star \AB_{C_1}  \ \ \simeq  \ \  \EB,
\]
where in the second homotopy equivalence we contracted the contractible complex $\PB_{C_1}\star \AB_{C_1}$.  This proves uniqueness of $\EB$, up to homotopy equivalence.  The stated properties of $\EB$ were proven in Lemma \ref{lemma:relprojector}.
\end{proof}

\begin{definition}\label{def:relproj}
We will write $\PB_{C_1}\leq \PB_{C_2}$ if either of the equivalent conditions of Theorem \ref{thm:rel idemp thm} are satisfied.  In this case the \emph{relative projector} $\EB\simeq \Cone(\PB_{C_1}\rightarrow \PB_{C_2})$  will be denoted $\PB_{C_2}/\PB_{C_1}$ or sometimes $\PB_{C_2/C_1}$.
\end{definition}

The following are some trivial observations.
\begin{observation}
The identity $\one\simeq \PB_{\one}$ is the unique maximum with respect to $\leq$ and $0=\PB_0$ is the unique minimum.  Moreover, $\one / \PB_C\simeq \AB_C$ and $\PB_C/0 = \PB_C$.
\end{observation}

\begin{observation}
If $C,D\in\AS$ are counital, then $\PB_C\leq \PB_{C\oplus D}$.
\end{observation}
\begin{observation}\label{obs:elim D}
If $C,D\in \AS$ are counital and $D\leq C$, then $\PB_C\simeq \PB_{C\oplus D}$.
\end{observation}

\begin{observation}
If $C$ contains $\one$ as a direct summand, then $\AB_C\simeq 0$ and $\PB_C\simeq \one$.
\end{observation}

\begin{observation}
If $(C,\e)$ is a counital object, then so is $(C\star C, \e\star \e)$, and $\PB_{C\star C}\simeq \PB_C$.
\end{observation}

\section{The normalized construction}
\label{s:normalized}
In this section we give a normalized version of the complexes $\PB_C$ and $\AB_C$.  We actually discovered these expressions first, only later realizing that they can be expressed in terms of equations \eqref{eq:introbarP} and \eqref{eq:introbarA}.

Ultimately we are seeking to simplify the complexes $\AB_C$ and $\PB_C$ by decomposing the chain groups $C^{\star n}$ as direct sums of simpler pieces.  The basic summands of $C^{\star n}$ are images of some distinguished idempotents $e_k\in C^{\star k}$ for various $k$. 

The main results can be summarized as follows.  Suppose $C\in \AS$ is counital and $\AS$ is idempotent complete.  Then:
\begin{itemize}
\item there exist objects $X_n\in \AS$ for $n\geq 1$ such that $C^{\star n}$ is a direct sum of $X_k$ with multiplicity $\binom{n-1}{k-1}$ for $1\leq k\leq n$ (in particular $X_n$ appears in $C^{\star n}$ with multiplicity one). 
\item there are canonical maps $\d_n:X_n\rightarrow X_{n-1}$ such that $\d_{n-1}\circ \d_n=0$ for all $n\geq 1$ (by convention we set $X_0:=\one$), and
\[
(C\rightarrow \underline{\one})^{\star n} \ \ \simeq \ \ (X_n \buildrel \d_n\over \longrightarrow \cdots \buildrel \d_2\over \longrightarrow X_1\buildrel \d_1\over \longrightarrow\underline{X_0}).
\]
\item there is a well-defined ``infinite tensor power'', which satisfies
\[
\text{``}(C\rightarrow \underline{\one})^{\star \infty}\text{''} \ \ \simeq \ \ ( \cdots \buildrel \d_2\over \longrightarrow X_1\buildrel \d_1\over \longrightarrow\underline{X_0}) \ \ \simeq \ \ \AB_C.
\]
Precisely speaking, the infinite tensor power is the homotopy colimit, or mapping telescope, of a directed system $\one\rightarrow \FB \rightarrow \FB^{\star 2}\rightarrow \cdots$, where $\FB=\Cone(C\buildrel\e\over\rightarrow \one)$.
\end{itemize}

\subsection{A distinguished direct summand of $C^{\star n}$}
\label{ss:en}
Let $\AS$ be a fixed additive monoidal $\k$-linear category, and let $C\in \AS$ be an object equipped with a map $\e:C\rightarrow \one$ and a map $\Delta:C\rightarrow C\star C$ satisfying the right counit axiom:
\begin{equation}\label{eq:rightcounit}
(\Id_C\star \e)\circ \Delta  = \Id_C.
\end{equation}
In this section we will not assume that $\Delta$ satisfies the left counit or coassociativity axioms.  Our goal in this section is to define and establish basic properties of some special idempotents $e_n\in \End_\AS(C^{\star n})$.

\begin{definition}\label{def:en}
Define endomorphisms $e_n\in \End_\AS(C^{\star n})$ for $n\geq 1$ as follows.  First, introduce shorthand $\Id_k := \Id_{C^{\star k}}$.  Let $e_1:=\Id_1$ and
\begin{equation}
e_2\ \ := \ \ \Id_2 - \Delta\circ (\Id_1\star \e),
\end{equation}
and for $n\geq 2$ set
\begin{equation}\label{eq:en}
e_n \ \ = \ \ (e_2\star \Id_{n-2})\circ (\Id_1\star e_2\star \Id_{n-3})\circ \cdots \circ (\Id_{n-2}\star e_2).
\end{equation}
By convention we set $C^{\star 0}=\one$ and $e_0:=\Id_0:=\Id_{\one}$.
\end{definition}

\begin{lemma}\label{lemma:en kills ep}
We have $(\Id_{l-1}\star \e \star \Id_{n-l})\circ e_n = 0$ for all $2\leq l\leq n$.
\end{lemma}
\begin{proof}
The statement is vacuous for $n=0,1$.  The right counit axiom \eqref{eq:rightcounit} implies that $(\Id_1\star \e)\circ e_2 = 0$ which proves the statement in case $n=2$.  The statement for $n>2$ follows from the definition of $e_n$ in terms of $e_2$ \eqref{eq:en}.
\end{proof}

\begin{lemma}[Left $e_n$ absorption]\label{lemma:e absorption}
Suppose $f:D\rightarrow E\star C^{\star n} \star F$ satisfies
\[
(\Id_E\star (\Id_{l-1}\star \e\star \Id_{n-l}) \star \Id_F)\circ f =  0
\]
for $2\leq l\leq n$.  Then 
\[
(\Id_E\star e_n\star \Id_F)\circ f =  f.
\]
\end{lemma}
\begin{proof}
The lemma is trivially true for $n=0,1$.  Observe that $\Id_E\star (\Id_{l-2}\star e_2\star \Id_{n-l}) \star \Id_F$ is  $\Id_{F\star C^{\star n}\star E}$ plus a morphism which, under the hypotheses on $f$, becomes zero after precomposing with $f$. Thus we have
\[
(\Id_E\star (\Id_{l-2}\star e_2\star \Id_{n-l}) \star \Id_F)\circ f = f
\]
for all $2\leq l\leq n$.  The lemma now follows from the definition of $e_n$.
\end{proof}

Combining the previous two lemmas shows that $e_n$ is idempotent.
\begin{lemma}
We have $e_n^2 = e_n$ for all $n\in \Z_{\geq 0}$.\qed
\end{lemma}

\subsubsection{Karoubi envelope}
\label{sss:karoubi}
If $\AS$ is idempotent complete then we have objects of $\AS$ corresponding to the images of the idempotents $e_n$.  If $\AS$ is not idempotent complete then we may embed $\AS$ fully faithfully into an idempotent complete category via the following well-known construction.

\begin{definition}\label{def:karoubi envelope}
If $\AS$ is an additive $\k$-linear category then the \emph{Karoubi envelope} of $\AS$ is the $\k$-linear category $\Kar(\AS)$ with objects pairs $(X,e)$ where $X\in \AS$ is an object and $e\in \End_\AS(X)$ is idempotent.  A morphism $(X_0,e_0)\rightarrow (X_1,e_1)$ in $\Kar(\AS)$ is a morphism $f:X_0\rightarrow X_1$ such that $e_1\circ f \circ e_0 = f$.
\end{definition}
The object $(X,e)\in \Kar(\AS)$ will also be denoted $\im e$. 
\begin{remark}
There is a canonical fully faithful functor $\AS\rightarrow \Kar(\AS)$ sending $X\mapsto (X,\Id_X)$.  We typically denote $(X,\Id_X)$ simply by $X$, and regard $\AS$ as a full subcategory of $\Kar(\AS)$.
\end{remark}
\begin{remark}  If $\AS$ is monoidal then so is $\Kar(\AS)$, in a natural way.  On the level of objects, this monoidal structure is defined by
\[
(X,e)\star (X',e') = (X\star X', e\star e').
\]
\end{remark}
\begin{remark}The category $\AS$ is idempotent complete if and only if the canonical functor $\AS\rightarrow \Kar(\AS)$ is an equivalence.
\end{remark}
\begin{remark}The idempotent $e\in \End_\AS(X)$ can be regarded as a morphism in $\Kar(\AS)$ in (at least) three ways: either as the identity endomorphism of $\im e$, or the inclusion / projection of $\im e$ as a direct summand of $X$.
\end{remark}

\subsubsection{Tensor structure}
\label{sss:tensoring e}
\begin{lemma}\label{lemma:C times im e}
Retain notation as in \S \ref{ss:en}.  Then we have an isomorphism
\[
C\star \im e_n \cong \im e_{n+1}\oplus \im e_n \qquad \qquad 
\]
in $\Kar(\AS)$, for $n\geq 1$.
\end{lemma}
\begin{proof}
Define morphisms $\pi:C^{\star n+1}\leftrightarrow C^{\star n}:\sigma$ by the formulas
\[
\pi \ := \ (\Id_1\star \e\star \Id_{n-1})\circ (\Id_1\star e_n),
\]
\[
\sigma \ := \ (\Delta\star \Id_{n-1})\circ e_n.
\]
Observe that $(\Id_{l-1}\star \e\star \Id_{n-l})\circ \pi = 0$ for $2\leq l\leq n$, hence $e_n\circ \pi = \pi$ by Lemma \ref{lemma:e absorption}.  Thus, $\pi$ can be regarded as a morphism $\im(\Id_1\star e_n)\rightarrow \im(e_n)$ in $\Kar(\AS)$.

Similarly $\sigma\circ e_n = \sigma$ is clear, and $(\Id_{l-1}\star \e\star \Id_{n+1-l})\circ \sigma =0$ for $3\leq l\leq n+1$ implies $(\Id_1\star e_n)\circ \sigma = \sigma$.  Thus, $\sigma$ may be regarded as a morphism $\im(e_n)\rightarrow \im(\Id_1\star e_n)$.

Compute:
\begin{eqnarray*}
\pi\circ \sigma & =& (\Id_1\star \e\star \Id_{n-1})\circ (\Id_1\star e_n)\circ \sigma\\
&=& (\Id_1\star \e\star \Id_{n-1})\circ \sigma\\
&=& (\Id_1\star \e\star \Id_{n-1})\circ (\Delta\star \Id_{n-1})\circ e_n\\
&=& e_n.
\end{eqnarray*}
It follows that $\sigma\circ \pi$ is an idempotent endomorphism of $\im(\Id_1\star e_n)$ whose image is isomorphic to $\im(e_n)$.  The complementary idempotent is
\begin{eqnarray*}
\Id_1\star e_n - \sigma\circ \pi 
& =& \Id_1\star e_n -  (\Delta\star \Id_{n-1})\circ e_n\circ \pi\\
& =& \Id_1\star e_n -  (\Delta\star \Id_{n-1})\circ \pi\\
& =& \Id_1\star e_n -  (\Delta\star \Id_{n-1})\circ (\Id_1\star \e\star \Id_{n-1})\circ (\Id_1\star e_n)\\
& =& (e_2\star \Id_{n-1})\circ (\Id_1\star e_n)\\
&=& e_{n+1}
\end{eqnarray*}
Thus, the image of $\Id_1\star e_n$ (that is to say, $C\star \im e_n$) is isomorphic to the direct sum of $\im e_n$ and $\im e_{n+1}$, as claimed.
\end{proof}

\begin{remark}
With only slightly more work, it is possible to show that $(\im e_n)\star (\im e_m)\cong \im e_{n+m}\oplus \im e_{n+m-1}$ for all $n,m\geq 1$.
\end{remark}

\subsection{The normalized categorical idempotents}
\label{ss:normalized idempts}

Recall that $(\Id_{l-1}\star \e \star \Id_{n-l})\circ e_n$ is zero for $2\leq l\leq n$.  When $l=1$ this morphism is generally nonzero.

\begin{definition}\label{def:del n}
Let $\d_n: \im e_n\rightarrow \im e_{n-1}$ be the morphism given by $\d_n:=(\e\star \Id_{n-1})\circ e_n$, for $n\geq 1$.
\end{definition}
Note that $e_{n-1}\circ \d_n = \d_n$ by an application of Lemma \ref{lemma:e absorption}, so $\d_n$ can indeed be regarded as a morphism $\im e_n \rightarrow \im e_{n-1}$, as claimed.

\begin{lemma}\label{lemma:del squared}
We have $\d_{n-1}\circ \d_n = 0$ for all $n\geq 2$.
\end{lemma}
\begin{proof}
Compute: $\d_{n-1}\circ \d_n = (\e\star \e \star \Id_{n-2})\circ e_n = 0$.
\end{proof}

\begin{definition}\label{def:normalized A}
Retain notation as in \S \ref{ss:en}. Let $\hat{\AB}_C,\hat{\PB}_C\in\KC^-(\Kar(\AS))$ denote the complexes
\[
\hat{\AB}_C \ \ = \ \ \cdots \buildrel \d_3 \over \longrightarrow  \im e_2 \buildrel \d_2 \over \longrightarrow \im e_1 \buildrel \d_1 \over \longrightarrow \underline{\one},
\]
\[
\hat{\PB}_C \ \ = \ \ \cdots \buildrel \d_4 \over \longrightarrow  \im e_4 \buildrel \d_3 \over \longrightarrow \im e_2 \buildrel \d_2 \over \longrightarrow \underline{\im e_1}.
\]
\end{definition}

Note that we do not yet assume $C$ is fully counital, only that $\Id_C\star \e$ admits a right inverese.
\begin{lemma}\label{lemma:normalized A kills C}
We have $C\star \hat{\AB}_C\simeq 0$.
\end{lemma}
\begin{proof}
By Lemma \ref{lemma:C times im e} the chain groups of $C\star \hat{\AB}_C$ split as direct sums:
\[
C\star \im e_n \cong \im e_{n+1}\oplus \im e_n.
\]
We claim that when expressed in terms of this direct sum decomposition the differential is
\begin{equation}\label{eq:deln matrix}
\Id_C\star \d_n \ = \ \sqmatrix{0&\Id_{\im e_n}\\0&0} \ : \im e_{n+1}\oplus \im e_n \rightarrow \im e_n \oplus \im e_{n-1}
\end{equation}
for $n\geq 2$.  The case $n=1$ is similar (ignore the $\im e_{n-1}$ summand and delete the bottom row of the above matrix).


First, observe that
\[
(\Id_1\star \d_n)\circ e_{n+1} = 0
\]
by Lemma \ref{lemma:en kills ep}, so the first column of the matrix for $\Id_1\star \d_n$ is zero, and
\[
e_n\circ (\Id_1\star \d_n) = (\Id_1\star \d_n),
\]
by Lemma \ref{lemma:e absorption}, so the second row of the matrix for $\Id_1\star \d_n$ is zero (the ``image'' of $\Id_1\star \d_n$ is completely contained $\im e_n$).  Thus, the matrix for $\Id_1\star \d_n$ has the form
\[
\Id_1\star \d_n \ = \ \sqmatrix{0&f\\0&0} \ : \im e_{n+1}\oplus \im e_n \rightarrow \im e_n \oplus \im e_{n-1}
\]
for some morphism $f:\im e_n \rightarrow \im e_n$.  To describe $f$ explicitly we use the explicit inclusion of the direct summand $\sigma:=(\Delta\star \Id_{n-1})\circ e_n:\im e_{n} \rightarrow \im \Id_1\star e_n$ from the proof of Lemma \ref{lemma:C times im e}.   Compute:
\[
f = (\Id_1\circ \d_n)\circ \sigma = (\Id_1\circ \e\star \Id_{n-1})\circ (\Delta\star \Id_{n-1})\circ e_n = e_n.
\]
This shows that $f$ is the identity of $\im e_n$ and proves \eqref{eq:deln matrix}.

Thus, all the terms summands in $C\star \hat{\AB}_C$ cancel in pairs and so $C\star \hat{\AB}_C\simeq 0$.
\end{proof}

\begin{theorem}\label{thm:normalized}
If $C\in \AS$ is counital then $\AB_C\simeq \hat{\AB}_C$ and $\PB_C\simeq \hat{\PB}_C$ (Definition \ref{def:normalized A}).
\end{theorem}
\begin{proof}
Since $\hat{\AB}_C$ kills $C$ from the left we have $\AB_C\star \hat{\AB}_C\simeq \hat{\AB}_C$.  Since $\AB_C$ kills $C$ from the right we have $\AB_C\star \hat{\AB}_C\simeq {\AB}_C$, hence $\AB_C\simeq \hat{\AB}_C$.  Then $\hat{\PB}_C\simeq \PB_C$ since an idempotent coalgebra is uniquely determined by its complement up to isomorphism.
\end{proof}

\subsubsection{The usual normalized two-sided bar complex}
\label{ss:usual normalized bar}
It is interesting to compare the abstract normalized complex $\hat{\AB}_C$ and its complement $\hat{\PB}_C$ with the usual normalized two-sided bar complex.

Let $A$ be a $\k$-algebra.  Let $\AS$ be the category of $A,A$-bimodules with tensor product $\star = \otimes_A$.  We will abbreviate by writing $\otimes=\otimes_\k$ (this is \emph{not} part of the monoidal structure on $\AS$).  The \emph{normalized} two-sided bar complex associated to $A$ is the complex
\[
\cdots \rightarrow A\otimes (A/\k)^{\otimes 2} \otimes A\rightarrow A\otimes (A/\k) \otimes A \rightarrow  \underline{A\otimes A}
\]
with the usual bar differential.

We claim that $A\otimes (A/\k)^{\otimes n-1}\otimes A$ is isomorphic to the image of the idempotent $e_n$ acting on $C^{\star n} = A^{\otimes n+1}$.

\begin{definition}  For each subset $S\subset \{1,\ldots,k\}$ let $f_S$ denote the bimodule endomorphism of $A\star A^{\otimes  k} \otimes A$ which ``shifts to the right the factors in positions $i\in S$''.   More precisely $f_S(a_0\otimes\cdots \otimes a_{k+1})$ is a simple tensor of the form $b_0\otimes \cdots\otimes b_{k+1}$ where
\begin{enumerate}
\item $b_0=a_0$.
\item if $i\not\in S$ and $i-1\not \in S$, then $b_i=a_i$.
\item if $i\in S$ and $i-1\not \in S$, then $b_i=1$.
\item if $i\not\in S$ and $i-1 \in S$, then $b_i=a_{i-1}a_i$.
\item if $i\in S$ and $i-1 \in S$, then $b_i=a_{i-1}$.
\end{enumerate}
\end{definition}

\begin{example}
If $k=5$ and $S=\{1,2,5\}$, then
\[
f_S(a_0\otimes \cdots\otimes a_6) = a_0\otimes 1\otimes a_1\otimes a_2a_3\otimes a_4\otimes 1\otimes a_5a_6.
\]
The original simple tensor can also be expressed as the component-wise product
\[
(1\otimes a_1\otimes a_2\otimes 1\otimes 1\otimes a_5\otimes 1)(a_0\otimes 1\otimes 1\otimes a_3\otimes a_4\otimes 1\otimes a_6),
\]
and the image under $f_S$ is visualized as shifting the first factor to the right, then multiplying:
\[
(1\otimes 1\otimes a_1\otimes a_2\otimes 1\otimes 1\otimes a_5)(a_0\otimes 1\otimes 1\otimes a_3\otimes a_4\otimes 1\otimes a_6).
\]
\end{example}

\begin{lemma}\label{lemma:E idemp}
The idempotent $e_{k+1}$ acting on $A\otimes A^{\otimes k}\otimes A$ satisfies
\begin{enumerate}
\item $e_{k+1}(a_0\otimes \cdots \otimes a_{k+1}) = \sum_{S\subset \{1,\ldots,k\}} (-1)^{|S|} f_S(a_0\otimes \cdots\otimes a_{k+1}).$
\item $e_{k+1}(a_0\otimes \cdots \otimes a_{k+1}) = a_0\otimes\cdots\otimes a_{k+1} + (\cdots)$, where $(\cdots)$ denotes a sum of simple tensors $b_0\otimes\cdots b_{k+1}$ with $b_i=1$ for some $1\leq i\leq k$.
\item if $a_i =1$ for some $1\leq i\leq k$, then $e_{k+1}(a_0\otimes \cdots \otimes a_{k+1})=0$.
\end{enumerate}
\end{lemma} 
\begin{proof}
Exercise.
\end{proof}

\begin{proposition}
The image of $e_{k+1}$ acting on $A\otimes A^{\otimes k}\otimes A$ is isomorphic to $A\otimes(A/\k)^{\otimes k}\otimes A$.
\end{proposition}
\begin{proof}
The idempotent $e_{k+1}$ annihilates all simple tensors in $a_0\otimes \cdots \otimes a_{k+1}$ in which $a_i=1$ for some $1\leq i\leq k$ by part (3) of Lemma \ref{lemma:E idemp}, so the projection $A\otimes A^{\otimes k}\otimes A\twoheadrightarrow \im (e_{k+1})$ descends to a surjective map $A\otimes(A/\k)^{\otimes k}\otimes A\twoheadrightarrow \im(e_{k+1})$.  We have to show that this map is injective.  If $e_{k+1}(z)=0$ then $z$ is a sum of simple tensors $a_0\otimes \cdots\otimes a_{k+1}$ with $a_i=1$ for some $1\leq i\leq k$, by part (2) of Lemma \ref{lemma:E idemp}.  Thus, $z$ is zero in $A\otimes(A/\k)^{\otimes k}\otimes A$.  This completes the proof.
\end{proof}

\begin{proposition}
The normalized complex $\hat{\PB}_{A\otimes_\k A}$ is isomorphic to the usual normalized two-sided bar complex.
\end{proposition}
\begin{proof}
As we've seen already, the chain groups of $\hat{\PB}_{A\otimes_\k A}$ satisfy $(\hat{\PB}_{A\otimes_\k A})^{-k} = A\otimes (A/\k)^{\otimes k}\otimes A$, which are the chain groups of the normalized two-sided bar complex as well.  To check that the differentials agree is an exercise.
\end{proof}

\subsubsection{Remark on Grothendieck groups}
\label{sss:K0}
This subsection is an informal discussion concerning Grothendieck groups.  We would like to consider the class of $\hat{\PB}_C$ in the Grothendieck group.  But since $\hat{\PB}_C$ is an infinite complex one encounters the usual problem that the relevant Grothendieck group is zero.  There are various ways around this problem in the examples of interest.  Throughout this (very informal and certainly incorrect as written) subsection we assume that such issues are dealt with, so we may consider the Euler characteristic of $[\hat{\PB}_C]$ as a well defined element of (an appropriate completion of) $K_0(\Kar(\AS))$.    See \cite{AchStr}.

Let $X_n:=\im e_n$ in $\Kar(\AS)$, for $n\geq 1$.  We have $X_1=C$ and generally $C\star X_n\cong X_{n+1}\oplus X_n$.  Thus, on the level of Grothendieck groups we have
\begin{equation}\label{eq:Xn in K0}
[X_n] = [C] ([C] - 1)^{n-1}
\end{equation}
for all $n\geq 1$ (by an easy induction).  Thus, the Euler characteristic of $\hat{\PB}_C$ is
\[
[\hat{\PB}_C] \ =\ \sum_{n\geq 1} (-1)^{n-1} [X_n] \ = \ [C] \sum_{n\geq 1} (-1)^{n-1}([C]-1)^{n-1}.
\]
One is certainly tempted to sum the geometric series, obtaining
\[
[\hat{\PB}_C] \ =\  \frac{[C]}{1+([C]-1)} \ = \ [\one],
\]
hence
\[
\hat{\AB}_C = 0.
\]

\renewcommand{\l}{\lambda}
The conclusion would then be that we are only able to categorify the most boring idempotents (zero and one)!  Thankfully, this conclusion is incorrect, essentially because $[C]$ may be a zero divisor in the Grothendieck group, hence completing with respect to the ideal generated by $([C]-1)$ is potentially a very destructive operation.

However, in some important examples, $C$ is quasi-idempotent in the sense that $C\star C \cong \l C$, where $\l$ is some ``scalar object'' (for instance a direct sum of copies of $\one$ with shifts, when this makes sense).  In this instance $C^2 = \l C$ implies that $X_n= (\l-1)^{n-1}C$, and we conclude that
\[
[\hat{\PB}_C] \ =\   \frac{[C]}{1+([\l]-1)}  = \frac{[C]}{[\l]}
\]
in $K_0(\AS)\llbracket [\l] -1 \rrbracket$.  The scalar object $\l$ is very often not a zero divisor in $K_0(\AS)$, so adjoining $[\l]\inv$ to $K_0(\AS)$ or completing with respect to the ideal generated by $[\l]-1$ is typically an innocuous operation (in contrast to adjoining $[C]\inv$).  In other words, $\hat{\PB}_C$ categorifies the idempotent obtained from $[C]$ by rescaling.

\begin{example}\label{ex:dividing by 2}
Let $R=\k[\a]$, regarded as a $\Z$-graded algebra, where $\a$ is a formal indeterminate of degree 2, and let $\AS$ be the additive monoidal category freely generated by the $\Z$-graded bimodules $\one=R$ and $C:=\k[\a]\otimes_{\k[\a^2]} \k[a]$ and their shifts.  This category is monoidal via $\star = \otimes_R$.

Then $C$ is a coalgebra object and satisfies $C\star C \cong (1+q^2)C$, where a polynomial in $q$ denotes the corresponding direct sum of copies of $\one$ with shifts.  One can then check directly that $X_n = q^{2n-2}C$, and the normalized complex $\hat{\PB}_C$ is given by
\[
\hat{\PB}_C \ \ = \ \ \cdots \rightarrow q^4 C \rightarrow q^2 C \rightarrow \underline{C}
\]
in which the maps alternate between multiplication by $\a\otimes 1 - 1\otimes \a$ and $\a\otimes 1+ 1\otimes \a$.  On the level of Grothendieck groups, this becomes
\[
[\hat{\PB}_C] \ = \ \frac{1}{1+q^2} [C].
\]
\end{example}

\subsection{The infinite power construction}
\label{ss:infinite powers}

Next we discuss the expression of categorical idempotents as ``infinite tensor powers'' of some given complex $\FB$. Let $C\in \AS$ be an object with counit $\e:C\rightarrow \one$ as above. Let $\FB:=\Cone(\e)$, which is the complex $C\rightarrow \underline{\one}$ with $\one$ in degree 0, $C$ in degree $-1$, and differential given by $\e$.  The inclusion of the degree zero chain object gives a chain map $\iota:\one\rightarrow \FB$, from which we may construct the following directed system
\begin{equation}\label{eq:directed system}
\begin{tikzpicture}[baseline=-.3em]
\node (a) at (0,0) {$\one$};
\node (b) at (2,0) {$\FB$};
\node (c) at (4,0) {$\FB^{\star 2}$};
\node (d) at (6,0) {$\cdots$};
\path[->,>=stealth',shorten >=1pt,auto,node distance=1.8cm]
(a) edge node {$\phi_0$} (b)
(b) edge node {$\phi_1$} (c)
(c) edge node {$\phi_2$} (d);
\end{tikzpicture},\qquad\quad \phi_k :=\iota\star \Id_{\FB}^{\star k}.
\end{equation}

\begin{definition}\label{def:F to the infty}
Let $\FB^{\star \infty}$ denote the homotopy colimit of the directed system \ref{eq:directed system}. 
\end{definition}
For a precise model for this homotopy colimit one may use the \emph{mapping telescope}.  A priori $\FB^{\star \infty}$ lives in $\KC^-(\AS')$ for some cocompletion $\AS'$ of $\AS$.  But as we will see $\FB^{\star \infty}$ is homotopy equivalent to a complex in $\KC^-(\AS)$.

\begin{lemma}\label{lemma:powers of F}
Let $\FB:=\Cone(\e)$ as above.  Then
\begin{equation}\label{eq:F simplified}
\FB^{\star n}  \ \ \simeq  \ \ (\im e_n \rightarrow \im e_{n-1}\rightarrow\cdots \rightarrow \im e_1 \rightarrow \underline{\one})
\end{equation}
in which the differential $\im e_k \rightarrow \im e_{k-1}$ is  $\d_k$ from Definition \ref{def:del n}.
\end{lemma}
\begin{proof}
We prove this by induction on $n\geq 1$.  The base case $n=1$ is trivially true.  Assume by induction that \eqref{eq:F simplified} holds.  Tensoring on the left with $C$ yields
\[
C\star \FB^{\star n} \ \ \simeq \ \ (C\star \im e_{n} \rightarrow \cdots \rightarrow C\star \im e_1 \rightarrow \underline{C}).
\]
Each chain group is isomorphic to $C\star \im e_{k} \cong \im e_{k+1} \oplus \im e_{k}$ and the differential is determined explicitly from Lemma \ref{lemma:normalized A kills C}.  The terms $\im e_k \rightarrow \im e_k$ cancel in pairs for $1\leq k\leq n$.  The only remaining term is $\im e_{n+1}$ in cohomological degree $-n$.  This shows that
\begin{equation}\label{eq:C times FFFF}
C \star \FB^{\star n} \simeq \im e_{n+1}[n].
\end{equation}
Now, since $\FB = \Cone(C\rightarrow \one)$, we can write $\FB^{\star n+1}$ as the cone of a map $C\star \FB^{\star n}\rightarrow \FB^{\star n}$ or, equivalently as the cone of a map
\[
\im e_{n+1}[n]\rightarrow (\im e_n \rightarrow \cdots \rightarrow \im e_1 \rightarrow \underline{\one}).
\]
Such a cone is necessarily a complex of the form
\[
(\im e_{n+1}\rightarrow \im e_n \rightarrow \cdots \rightarrow \im e_1 \rightarrow \underline{\one}).
\]
To complete the computation we must check that the new (leftmost) component of the differential is as claimed.  This component is the idempotent $e_{n+1}$, regarded as the inclusion $\im e_{n+1}\rightarrow C\star \im e_n$ followed by the appropriate component of the differential $\e\star \Id_n : C\star F^{\star n} \rightarrow F^{\star n}$.  This composition is $\d_{n+1}$ as claimed.  This completes the proof.
\end{proof}

\begin{remark}
The proof actually establishes something stronger, namely that $\iota\star \Id_{\FB^{\star k}}:\FB^{\star k} \rightarrow \FB^{\star k+1}$ corresponds under \eqref{eq:F simplified} to the obvious inclusion of a subcomplex.
\end{remark}

The following is now an easy corollary.
\begin{theorem}\label{thm:powers}
We have $\FB^{\star \infty}\simeq \hat{\AB}_C$.  In particular if $C$ is counital then $\AB_C\simeq \hat{\AB}_C\simeq \FB^{\star \infty}$.\qed
\end{theorem}

\section{Examples}
\label{s:examples}

In \S \ref{ss:recognition} we establish a result on how to recognize when an idempotent (co)algebra in $\KC^-(\AS)$ is isomorphic to $\PB_C$ or $\AB_C$ for some $C$ (Lemma \ref{lemma:recognition}) , which may be useful in future applications.  In \S \ref{ss:cat TL}  we show how to obtain the categorified Temperley-Lieb idempotents from \cite{CH12} using techniques in this paper.

\subsection{Recognizing idempotents of the form $\PB_C$ and $\AB_C$}
\label{ss:recognition}


\begin{lemma}
Let $\PB$ be an idempotent coalgebra in $\Ch^-(\AS)$ and $\AB=\Cone(\PB\buildrel\boldsymbol{\e}\over\rightarrow \one)$ its complement.  Suppose $X,Y\in \Ch^-(\AS)$ satisfy $\PB\star X\simeq X$ and $\AB\star Y\simeq Y$.  Then
\[
\Hom_{\Ch(\AS)}(X,Y)\simeq 0
\]
for all $X,Y\in \Ch^-(\AS)$.
\end{lemma}
For experts: in the statement above the category $\Ch^-(\AS)$ can be replaced by any pretriangulated dg monoidal category.
\begin{proof}
From the hypotheses, we have $\AB\star X\simeq 0$, i.e.~ $\Cone(\boldsymbol{\e})\star X\simeq 0$, hence $\boldsymbol{\e}\star \Id_X$ gives a homotopy equivalence $\PB\star X\rightarrow \one\star X\cong X$.  Similarly, the unit map $\boldsymbol{\eta}:\one\rightarrow \AB$ gives a homotopy equivalence $Y\rightarrow \AB\star Y$.  Pre- and post-composing with these homotopy equivalences defines a homotopy equivalence
\begin{equation}\label{eq:he who is zero}
\Hom_{\Ch(\AS)}(X,Y)\rightarrow \Hom_{\Ch(\AS)}(\PB\star X,\AB \star Y). 
\end{equation}
This homotopy equivalence sends $f\in \Hom_{\Ch(\AS)}(X,Y)$ to $(\boldsymbol{\eta}\circ \boldsymbol{\e})\star f$.  But $\boldsymbol{\eta}\circ \boldsymbol{\e}$ is null-homotopic, being the composition of canonical maps associated to the mapping cone: $\PB\buildrel\boldsymbol{\e}\over\rightarrow \one \rightarrow \Cone(\boldsymbol{\e})$.  It follows that \eqref{eq:he who is zero} is null-homotopic, hence both source and target are contractible complexes.
\end{proof}

It follows that $\Hom_{\Ch(\AS)}(X,\AB\star Y)\simeq 0$ for all $X\in \Ch^-(\AS)$ satisfying $\PB\star X\simeq X$ (with $Y\in \Ch^-(\AS)$ arbitrary).  Consequently $\Hom_{\Ch(\AS)}(X,\PB\star Y)\simeq \Hom_{\Ch(\AS)}(X,Y)$ for all such $X$.  This will be used in the proof of the following.

\begin{lemma}\label{lemma:recognition}
Let $\AS$ be a $\k$-linear additive idempotent complete monoidal category.  Let $\PB$ be an idempotent coalgebra in $\Ch^-(\AS)$ and $\AB=\Cone(\PB\rightarrow \one)$ its complement.  Suppose that $\BS\subset \AS$ is a subcategory such that $\PB\in \Ch^-(\BS)$ and $X\star \AB\simeq 0 \simeq \AB\star X$ for all $X\in \BS$.  Then
\begin{enumerate}
\item The complexes $\PB$ and $\AB$ are homotopy equivalent to complexes supported in non-positive homological degrees.
\item Assuming as in (1) that $\PB$ and $\AB$ are supported in non-positive homological degrees, the object $C=\PB^0$ is counital and $\PB\simeq \PB_C$, $\AB\simeq \AB_C$.
\end{enumerate}
\end{lemma}
\begin{proof}
Let $N$ be the largest integer for which the chain object $\PB^N$ is nonzero.  Since $\PB^N\in \BS$ we have $\AB\star \PB^N\simeq 0$ hence $\PB\star \PB^N\simeq \PB^N$.  Then we compute the complex of homs
\[
\Hom_{\Ch(\AS)}(\PB^N, \PB)\simeq \Hom_{\Ch(\AS)}(\PB^N, \one) = \Hom_{\AS}(\PB^N,\one),
\]
which is supported in degree zero.  The inclusion of $\PB^N[-N]$ into $\PB$ can be though of as a degree $N$ closed element of the above hom complex.  If $N>0$ then this morphism must be null-homotopic, which implies that $\PB^N$ can be cancelled with a Gaussian elimination.  In degree $N-1$ the result of Gaussian elimination replaces $\PB^{N-1}$ with a direct summand of itself, but this summand is still annihilated by $\AB$, so if $N-1>0$ then this new term can be cancelled by the same argument.  Continuing in this fashion, we obtain a a complex which is homotopy equivalent to $\PB$, supported in non-positive homological degrees.  This proves (1).

Now, assume that the chain objects $\PB^k$ are zero for $k>0$, and let $C:=\PB^0$ be the chain object in degree zero.  Let $\e:C\rightarrow \one$ be the zeroth (and only nontrivial) component of the counit map $\PB\rightarrow \one$.

Then $\e$ is a counit by considering the right-most components of null-homotopies for $C\star \AB\simeq 0 \simeq \AB\star C$.  Statement (2) now follows from the uniqueness statement for $\PB_C$ and $\AB_C$.
\end{proof}

\begin{remark}
Note the asymmetry in the above statement: there is a dual version of Lemma \ref{lemma:recognition} in which the roles of unital and counital idempotents are reversed, \emph{provided that we also replace $\Ch^-(\AS)$ with $\Ch^+(\AS)$}.  
\end{remark}
\subsection{Categorification of Temperley-Lieb idempotents}
\label{ss:cat TL}
In this section we assume familiarity with \cite{B-N05,Kh02}, which describe a categorification of Temperley-Lieb algebras.

\begin{remark}
We will use the ``dotted cobordisms'' version of Bar-Natan's categories.  The details don't concern us here, and we refer to \cite{B-N05} for details (see also \S2.2 of \cite{quasiPaper} for a recap).
\end{remark}

For integers $n,m\geq 0$ of the same parity, let $\TLC_{m,n}'$ denote Bar-Natan's category of $m,n$-tangles and cobordisms.  An object of this category is an embedded 1-submanifold of $I\times I$, with boundary $(D_m\times\{1\})\cup (D_n\times\{0\})$, where $D_m\subset (0,1)$ is some prescribed set of $m$ points.  A morphism $T\rightarrow T'$ in this category is a formal $\Z$-linear combination of cobordisms in $I\times I\times I$, modulo relations.  See, e.g.~ Definition 2.3 in \cite{quasiPaper} for the precise relations.  The morphism spaces in $\TLC_{m,n}$ are graded by declaring the degree of a cobordism $\Sigma:T_0\rightarrow T_1$ to be
\[
\deg(\Sigma) = \#\{\text{saddle points}\} - \#\{\text{local maxima and minima}\} +2\#\{\text{dots}\}.
\]

Then we let $\TLC_{m,n}$ be the category obtained from the $\Z$-graded category $\TLC_{m,n}'$ by formally adjoining grading shifts and finite direct sums objects, denoted $\bigoplus_i q^{k_i} T_i$.  The convention for grading shifts is that a cobordism $\Sigma:T_0\rightarrow T_1$ with a single saddle point yields a degree zero map $q T_0\rightarrow T_1$.

The operation of composing tangles defines functors $\star : \TLC_{m,k}\otimes \TLC_{k,n}\rightarrow \TLC_{m,n}$, making the collection of categories $\TLC_{m,n}$ into a 2-category (the 1-morphism categories of which are the $\TLC_{m,n}$).
\begin{remark}
Alternately, we could define $\TLC_{m,n}$ to be the category of finitely generated graded projective modules over Khovanov's ring $H^{(m+n)/2}$.  Then the composition of tangles $\TLC_{m,k}\otimes \TLC_{k,n}\rightarrow \TLC_{m,n}$ corresponds to induction from $H^{(m+k)/2)}\otimes H^{(k+n)/2}$ to $H^{k+(m+n)/2}$, followed by tensoring with a special bimodule over $H^{(n+m)/2}, H^{k+(n+m)/2}$.
\end{remark}

There is a contravariant duality functor $(-)^\vee:\TLC_{m,n}\rightarrow \TLC_{n,m}$ which on the level of objects applies the transformation $(x,y)\mapsto (x,1-y)$ to all tangles, and on the level of morphisms applies $(x,y,z)\mapsto (x,1-y,1-z)$ to all cobordisms.  Then $q^{(n-m)/2} X^\vee$ is the right dual of $X$ in the usual sense of 2-categories.  Since $(X^\vee)^\vee=X$ one can say that $X$ and $X^\vee$ are biadjoint up to shift.

For each $1\leq i\leq n-1$ there is a distinguished tangle $T_{i,n}\in \TLC_{n,n-2}$ (the ``cup'').    A  \emph{cup tangle} is by definition a composition of these tangles.  Up to isotopy, these are indexed by certain binary sequences.

\begin{definition}
We put a partial order on $\{1,-1\}^n$ by declaring that two sequences sequences $\e=(\e_1,\ldots,\e_n)$ and $\nu=(\nu_1,\ldots,\nu_n)$ satisfy $\e\unlhd \nu$ if $\e_1+\cdots+\e_i \leq \nu_1+\cdots+\nu_i$ for all $1\leq i\leq n$.  A sequence $\e\in \{1,-1\}^n$ is called \emph{admissible} if $\e\unrhd 0$.

Given $\e\in \{1,-1\}$ we let $|\e|=\e_1+\cdots+\e_n$ and $r(\e) = \frac{1}{2}(n-|\e|)$.
\end{definition}

Associated to each admissible sequence $\e\in \{1,-1\}^n$ with $|\e|=k$ we have a cup tangle $T_\e\in \TLC_{n,k}$ defined as in the following example:\vskip7pt
\begin{equation}\label{eq:cupdiag}
T_\e \ \ = \ \  \begin{minipage}{2.3in}
\labellist
\small
\pinlabel $+$ at 2 65
\pinlabel $+$ at 16 65
\pinlabel $-$ at 32 65
\pinlabel $+$ at 48 65
\pinlabel $+$ at 64 65
\pinlabel $+$ at 80 65
\pinlabel $-$ at 96 65
\pinlabel $+$ at 112 65
\pinlabel $-$ at 128 65
\pinlabel $-$ at 144 65
\pinlabel $+$ at 160 65
\pinlabel $+$ at 177 65
\endlabellist
\includegraphics[scale=.9]{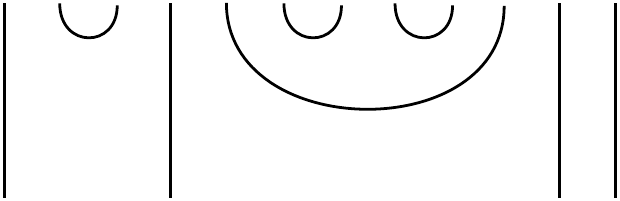}
\end{minipage},
\end{equation}
for which $\e = (1,1,-1,1,1,1,-1,1,-1,-1,1,1)$.  The strands in such a diagram are referred to as \emph{through strands} if they pass from the top to the bottom and \emph{turn-back strands} otherwise.   
Note that $r(\e)$ is the number of occurrences of $-1$ in $\e$, equivalently the number of turnback strands in $T_\e$.

To obtain the sequence of $\pm 1$ associated to a cup diagram as in \eqref{eq:cupdiag}, orient the ``through strands'' upward and the ``turn-back strands'' leftward.  Then one places a $+1$ at each outgoing point of the boundary and a $-1$ at each incoming point of the boundary, and reads along the top of the diagram  to obtain $\e$.

\begin{remark}
The number $r(\e)$ counts the turnback strands in $T_\e$.
\end{remark}
\begin{remark}\label{rmk:TL coalg}
The object $C_\e:=q^{r(\e)}T_\e\star T_\e^\vee$ comes with a canonical degree zero cobordism (sequence of $r(\e)$ saddle cobordisms)  $C_\e\rightarrow \one$ making $C_\e$ into a counital object (in fact a coalgebra object) in $\TLC_{n,n}$.
\end{remark}

\begin{definition}\label{def:Cnk}
For each pair of integers $0\leq k\leq n$ with $n-k$ even, let $C_{n,k}:=\bigoplus_{|\e|=k} C_\e$, where the sum is over admissible sequences $\e\in \{1,-1\}^n$ with $|\e|=k$.
\end{definition}

The \emph{through-degree} of a tangle $T\in \TLC_{n,m}$ is the minimal $k$ such that $T$ factors as $T \cong U\star U'$ with $U\in \TLC_{n,k}$ and $U'\in \TLC_{k,m}$.  Note for each $0\leq k\leq n$ the tangles with through degrees $\leq k$ (and direct sums of shifts thereof) form a two-sided tensor ideal in $\TLC_{n,n}$.

\begin{lemma}\label{lemma:throughdeg}
A tangle $T$ has through-degree $\leq k$ if and only if $T$ is left (equivalently right) $C_{n,k}$-projective (Definition \ref{def:Cproj}).
\end{lemma}
\begin{proof}
We only prove the statement about left $C_{n,k}$-projectives, since the statement for right $C_{n,k}$-projectives follows by symmetry.  We may assume without loss of generality that $T$ has no closed loop components.  Then $T$ can be decomposed into its ``top half'' and ``bottom half'' as $T = T_\e\circ T_\nu^\vee$ for some admissible sequences $\e,\nu\in \{1,-1\}^n$ with $|\e|=|\nu|=\text{through-degree of $T$}$.  Then $T$ is left $C_\e$ projective.

We must show that $C_\e\leq C_{n,k}$.  If $T$ has through-degree $k$, then $|\e|=k$ and $C_\e$ is a direct summand of $C_{n,k}$, so $C_\e\leq C_{n,k}$ is obvious.  Otherwise, if $|\e|=l<k$, then the counit of $C_{\e}$ (a sequence of saddle cobordisms each of which increases through-degree by 2) factors through some object $C_\mu$ with $|\mu|=k$, and $C_\e\leq C_\mu \leq C_{n,k}$, as claimed.  This shows that if the through-degree of $T$ is $\leq k$, then $T$ is left $C_{n,k}$-projective.  The converse is obvious, since each summand of $C_{n,k}$ has through-degree $k$.

\end{proof}


\begin{definition}\label{def:Pnk}
For each pair of integers $0\leq k\leq n$ with $n-k$ even, let $\PB_{n,k}:=\PB_{C_{n,k}}$ and $\EB_{n,k}:=\PB_{n,k}/\PB_{n,k-2}$ be the relative idempotent.  If $k-2<0$, then we set $\PB_{n,k}=0$ by convention, so that $\EB_{n,1}=\PB_{n,1}$ if $n$ is odd and $\EB_{n,0}=\PB_{n,0}$ if $n$ is even.
\end{definition}
\begin{remark}
$\PB_{n,n}\simeq\one$, so $\EB_{n,n} \simeq \one / \PB_{C_{n,n-2}} = \AB_{C_{n,n-2}}$.
\end{remark}

\begin{theorem}\label{thm:TL P and A}
The complexes $\EB_{n,k}$ from Definition \ref{def:Pnk} satisfy:
\begin{enumerate}
\item $\EB_{n,k}$ is a complex constructed from tangles with through-degree $\leq k$.
\item $X\star \EB_{n,k}\simeq 0 \simeq \EB_{n,k}\star X$ for any tangle $X$ with through-degree $<k$.
\item there is a homotopy equivalence $\one \ \ \simeq \ \ \tw_\a\left(\bigoplus_k \EB_{n,k}\right)$  where $\a$ is a twist whose component $\a_{l,k}\in \Hom^1_{\Ch(\TLC_{n,n})}(\EB_{n,k},\EB_{n,l})$ vanishes unless $k>l$.
\end{enumerate}
\end{theorem}
\begin{proof}
We have a homotopy equivalence of the form
\[
\one \ \ \simeq \ \ \PB_{n,n} \ \ \simeq \ \ \left(
\begin{tikzpicture}[baseline=3em]
\tikzstyle{every node}=[font=\small]
\node (a1) at (0,2) {$\PB_{n,n-2}[1]$};
\node (a0) at (0,0) {$\PB_{n,n}$};
\node (b1) at (2,2) {$\PB_{n,n-4}[1]$};
\node (b0) at (2,0) {$\PB_{n,n-2}$};
\node (c1) at (4,2) {$\ \ \cdots\cdots \ \ $};
\node (c0) at (4,0) {$\ \ \cdots\cdots \ \ $};
\node (e1) at (6,2) {$\PB_{n,0\text{ or }1}[1]$};
\node (e0) at (6,0) {$\PB_{n,2\text{ or }3}$};
\node (f) at (8,0) {$\PB_{n,0\text{ or }1}$};
\path[->,>=stealth',shorten >=1pt,auto,node distance=1.8cm]
(a1) edge node {} (a0)
(a1) edge node[above=.4em] {$-\Id$} (b0)
(b1) edge node {} (b0)
(b1) edge node[above=.4em] {$-\Id$} (c0)
(c1) edge node[above=.4em] {$-\Id$} (e0)
(e1) edge node {} (e0)
(e1) edge node[above=.4em] {$-\Id$} (f);
\end{tikzpicture}
\right)
\]
in which the right-hand side is a twisted complex of the form $\tw_\a(\bigoplus_k \PB_{n,k}\oplus \bigoplus_{k\neq n}\PB_{n,k}[1])$.  The vertical arrows are the maps which are guaranteed by $\PB_{n,k-2}\leq \PB_{n,k}$, the cones of which are $\EB_{n,k}$.  After reassociating we obtain a homotopy equivalence
\[
\one \ \ \simeq \ \ \left(
\begin{tikzpicture}[baseline=0em]
\tikzstyle{every node}=[font=\small]
\node (a0) at (0,0) {$\EB_{n,n}$};
\node (b0) at (2,0) {$\EB_{n,n-2}$};
\node (c0) at (4.5,0) {$\ \ \cdots\cdots \ \ $};
\node (e0) at (7,0) {$\EB_{n,0\text{ or }1}$};
\path[->,>=stealth',shorten >=1pt,auto,node distance=1.8cm]
(a0) edge node {} (b0)
(b0) edge node {} (c0)
(c0) edge node {} (e0);
\end{tikzpicture}
\right),
\]
which is a twisted complex as in (3).  Properties (1) and (2) of $\EB_{n,k}$ are easily verified.
\end{proof}

\begin{remark}
The complexes $\EB_{n,k}$ are homotopy equivalent to the complexes denoted $P_{n,k}^\vee$ in the notation of \cite{CH12} (where $(-)^\vee$ denotes the duality functor; this appears because of the preference in \emph{loc.~cit.} for complexes which live in $\Ch^+(\TLC_{n,n})$).  We use the letter ``$E$'' for these complexes because we prefer to reserve the letter `$P$' for idempotent coalgebras.
\end{remark}

\begin{remark}
The paper \cite{CH12} constructs a finer collection of mutually orthogonal idempotent complexes in $\Ch^-(\TLC_{n,n})$ which we denote here by $\EB_\e$, indexed by admissible sequences $\e\in \{1,-1\}^n$.  These can be obtained in a similar fashion as $\PB_{\unlhd\e}/\PB_{\lhd\e}$ where $C_{\unlhd \e}= \bigoplus_{\nu\unlhd \e} C_\nu$ and $\PB_{\unlhd \e}=\PB_{C_{\unlhd \e}}$, and similarly for $\PB_{\lhd \e}$.  We leave the details to the reader.
\end{remark}

\begin{remark}
It is possible to construct idempotent complexes of Soergel bimodules, for instance those constructed using categorical diagonalization \cite{ElHogCatDiag-pp,ElHogFTdiag-pp} using the techniques of this paper; we plan to address this in future work.
\end{remark}

\printbibliography

\end{document}